\documentclass[11pt,a4paper]{article}

\usepackage[margin=2.5cm]{geometry}

\usepackage[latin1]{inputenc}
\usepackage{amsmath,amssymb,amsthm}

\usepackage{bbm}
\usepackage{enumerate}

\usepackage[colorlinks=true,citecolor=black,linkcolor=black,urlcolor=blue]{hyperref}
\usepackage{hyphenat}

\usepackage{overpic}
\usepackage{graphicx}
\usepackage{wrapfig}
\usepackage{rotating}

\usepackage{chngcntr}

\usepackage{marginnote}

\usepackage[numbers]{natbib}
\usepackage[font=small,labelfont=bf]{caption}

\usepackage{wrapfig}

\newtheorem{theorem}{Theorem}
\newtheorem*{theorem*}{Theorem}
\newtheorem{lemma}[theorem]{Lemma}

\newtheorem{corollary}[theorem]{Corollary}

\newtheorem*{claim*}{Claim}

\newcommand{\E}{\mathbb{E}}
\newcommand{\Pb}{\mathbb{P}}

\newcommand{\N}{\mathbb{N}}
\newcommand{\R}{\mathbb{R}}

\newcommand{\Gnp}{\mathcal{G}(n,p)}
\newcommand{\Gnm}{\mathcal{G}(n,m)}

\newcommand{\mc}{\mathcal}
\newcommand{\parenth}[1]{\left(#1\right)}

\newcommand{\Deltag}{\Delta}

\newcommand{\ceilnk}{{\left\lceil \frac{n}{k}\right\rceil}}
\newcommand{\floornk}{{\left\lfloor \frac{n}{k}\right\rfloor}}
\newcommand{\bfr}{\mathbf{r}}

\newcommand{\constcone}{{100}}
\newcommand{\ftn}{\theta}

 \title{The chromatic number of dense random graphs}

\renewcommand{\le}{\leqslant}
\renewcommand{\ge}{\geqslant}
\renewcommand{\epsilon}{\varepsilon}

\renewcommand{\marginnote}[2][]{}

\author{Annika Heckel
\thanks{Mathematical Institute, University of Oxford,
Andrew Wiles Building, Woodstock Road, Oxford OX2~6GG, UK. E-mail: 
\texttt{heckel@maths.ox.ac.uk}
}
}
\date{\today 
}

\begin{document}

\maketitle
\begin{abstract}
The chromatic number $\chi(G)$ of a graph $G$ is defined as the minimum number of colours required for a vertex colouring where no two adjacent vertices are coloured the same. The chromatic number of the dense random graph $G \sim\Gnp$ where $p \in (0,1)$ is constant has been intensively studied since the 1970s, and a landmark result by Bollob\'as in 1987 first established the asymptotic value of $\chi(G)$ \cite{bollobas1988chromatic}. Despite several improvements of this result, the exact value of $\chi(G)$ remains open. In this paper, new upper and lower bounds for $\chi(G)$ are established. These bounds are the first ones that match each other up to a term of size $o(1)$ in the denominator: they narrow down the colouring rate $n/\chi(G)$ of $G \sim \Gnp$ to an explicit interval of length $o(1)$, answering a question of Kang and McDiarmid \cite{mcdiarmidsurvey}.
 \end{abstract}

\section{Introduction and results}

For $p \in [0,1]$, we denote by $G \sim \Gnp$ the \emph{Erd\H{o}s--R{\'e}nyi random graph} with $n$ labelled vertices where each of the $\binom{n}{2}$ possible edges is present independently with probability $p$. The \emph{chromatic number} $\chi(G)$ of a graph $G$ is defined as the minimum number of colours required for a \emph{proper colouring} of the vertices of $G$, which is a colouring where no two adjacent vertices are coloured the same. The chromatic number is one of the central topics both in graph theory in general and in the study of random graphs in particular, and has a wide range of applications including scheduling and resource allocation problems.

We say that an event $E=E(n)$ holds \emph{with high probability} (whp) if $\lim_{n \rightarrow \infty} \Pb(E) =1$. For two functions $f, g: \N \rightarrow \R$, we write $f = o(g)$ if $f(n)/g(n)\rightarrow 0$ as $n \rightarrow \infty$.  
The order of magnitude of the chromatic number of the dense random graph $G\sim \Gnp$ with constant $p\in (0,1)$ was first established by Grimmett and McDiarmid in 1975, who showed that whp,
\[
 (1+o(1)) \frac{n}{2\log_b n} \le \chi(G) \le  (1+o(1)) \frac{n}{\log_b n},
\]
where $b=\frac{1}{1-p}$. They also conjectured that the asymptotic value of $\chi(G)$ lies near the lower bound. Establishing the asymptotic behaviour of $\chi(G)$ remained one of the major open problems in random graph theory until it was settled by a breakthrough result of Bollob\'as in 1987 \cite{bollobas1988chromatic}, who showed that whp,
\[
 \chi(G) = (1+o(1)) \frac{n}{2\log_b n}.
\]
The same result was obtained independently by Matula and Ku\v cera \cite{matula:exposeandmerge}.

Refining Bollob\'as' approach, more accurate bounds were given by McDiarmid \cite{mcdiarmid1989method,mcdiarmid:chromatic}, who showed in particular that whp,
\[
 \chi(G) = \frac{n}{2 \log_b n -2 \log_b \log_b n+O(1)}.
\]
The current best upper bound was obtained by Fountoulakis, Kang and McDiarmid \cite{fountoulakis2010t} through a very accurate analysis of Bollob\'as' general approach, whereas the best lower bound comes from a first moment argument due to Panagiotou and Steger \cite{panagiotou2009note}: let  
\begin{equation}
\gamma = \gamma_p(n)= 2 \log_b n - 2 \log_b \log_b n - 2 \log_b 2, \label{defofgamma}
\end{equation}
then whp,
\begin{equation}
 \label{bestknownbounds}
\frac{n}{\gamma+o(1)} \le  \chi(G) \le \frac{n}{\gamma-1+o(1)}.
\end{equation}
As observed in \cite{fountoulakis2010t}, considering the above in terms of the \emph{colouring rate} $\bar \alpha(G)= n / \chi(G)$, which is the average colour class size of a proper colouring with the minimum number of colours, these inequalities give an explicit interval of length $1+o(1)$ which contains $\bar \alpha(G)$ whp. In \cite{mcdiarmidsurvey}, Kang and McDiarmid remark that it is a natural problem to determine the value of $\bar \alpha(G)$ up to an error of size $o(1)$.

The following result settles this question, giving new upper and lower bounds for $\chi(G)$ which match up to the $o(1)$ term in the denominator.

\begin{theorem}
\label{maintheorem}
Let $p\in(0,1)$ be constant, and consider the random graph $G \sim \mathcal{G}(n,p)$. Let $q=1-p$, $b=\frac{1}{q}$, $\gamma=\gamma_p(n)=2 \log_b n - 2 \log_b \log_b n - 2 \log_b 2$ and $\Deltag= \Deltag_p(n) = \gamma-\left\lfloor \gamma \right \rfloor$. Then whp,
\begin{equation*}
 \chi(G) = \frac{n}{\gamma-x_0+o(1)},
\end{equation*}
where $x_0\ge 0$ is the smallest nonnegative solution of
\begin{equation}
\label{aste}
 (1-\Deltag+x) \log_b (1-\Deltag+x)+\frac{(\Deltag-x)(1-\Deltag)}{2} \le 0.
\end{equation}
\end{theorem}
As $\Delta \ge 0$ is a solution of (\ref{aste}), $x_0$ is well-defined and $0 \le x_0 \le \Delta$. We will see in Lemmas \ref{corollary0} and \ref{corollarybounds} that for $p \le 1-1/e^2$, the smallest nonnegative solution of (\ref{aste}) is $x_0=0$, while for $p>1-1/e^2$, the solutions of (\ref{aste}) depend not only on $p$ but also on $n$, and we have $0 \le x_0 \le 1-\frac{2}{\log b}$ (in fact, the values of $x_0$ are dense in the interval $[0,1-\frac{2}{\log b}]$). Therefore, we can derive the following simpler bounds.
\begin{corollary}
\label{maincorollary}
Let  $p \in (0,1)$ be constant, and define $b$ and $\gamma$ as in Theorem \ref{maintheorem}. Consider the random graph $G \sim \mathcal{G}(n,p)$.
\begin{enumerate}[a)]
 \item If $p\le 1-1/e^2$, then whp,
\begin{equation*}
\chi(G) = \frac{n}{\gamma+o(1)} .
\end{equation*}
\item If $p>1-1/e^2$, then whp,
\[
 \frac{n}{\gamma+o(1)} \le \chi(G) \le \frac{n}{\gamma-1+\frac{2}{\log b}+o(1)}.
\]
\end{enumerate}
\end{corollary}
For $p\le 1-1/e^2$, the lower bound in Theorem \ref{maintheorem} is simply the known lower bound (\ref{bestknownbounds}) due to Panagiotou and Steger, which was obtained by estimating the first moment of the number of vertex partitions which induce proper colourings. The \emph{first moment threshold} of this random variable, i.e., the point where the first moment changes from tending to $0$ to tending to $\infty$, occurs at about $\frac{n}{\gamma+o(1)}$ colours.

For $p >1-1/e^2$, we shall also employ the first moment method to establish our new lower bound, although a different first moment threshold will take precedence. The \emph{independence number} $\alpha(G)$ is defined as the size of the largest \emph{independent set} in~$G$, i.e., the largest set of vertices without any edges between them. For $G\sim\Gnp$ with $p$ constant, $\alpha(G)$ takes one of at most two explicitly known consecutive values whp (for more details see Section \ref{overviewsection}). In a proper colouring, each colour class forms an independent set, and so no colour class can be larger than $\alpha(G)$. It will turn out that for $p>1-1/e^2$, $\gamma$ is so close to the likely values of $\alpha(G)$ that the hardest part in colouring $G$ is finding a sufficient number of disjoint independent sets of size $\left \lceil \gamma \right \rceil$ or larger. If we colour $G$ with about $\frac{n}{\gamma-x}$ colours for some $x>0$, then the average colour class size is about $\gamma-x$. If the independence number $\alpha(G)$ takes one of its likely values, then every such colouring must contain a partial colouring of a certain size consisting only of colour classes with at least $\gamma$ vertices. Condition (\ref{aste}) describes the first moment threshold of the number of such partial colourings.

The upper bound in Theorem \ref{maintheorem} is much harder to prove. In contrast to previous upper bounds, it will not be obtained through a variant of Bollob\'as' method but through the second moment method and our approach will be outlined in Section \ref{overviewsection}.

Analysing the second moment of the number of colourings of a random graph is a notoriously hard problem, as it involves examining the joint behaviour of all pairs of possible colourings, which varies considerably depending on how similar they are to each other. It has been previously studied in the sparse case where $p(n)$ tends to $0$ sufficiently quickly. Most notably, for $p=d/n$ where $d$ is constant, Achlioptas and Naor \cite{achlioptas2005two} used the second moment method to give two explicit values which the chromatic number of $\Gnp$ may take whp, and determined the chromatic number exactly for roughly half of all values $d$. Recently, Coja-Oghlan and Vilenchik \cite{coja2015chromatic} extended this result to almost all constant values $d$. For $p=n^{-c}$ where $3/4< c \le1$, Coja-Oghlan, Panagiotou and Steger \cite{coja2007chromatic} gave three explicit values for the chromatic number. In the dense case, however, the situation is quite different because the number of colours is much larger: in~\cite{achlioptas2005two}, the chromatic number is of order $O(1)$, whereas in our setting it is of order $\Theta(n/\log n)$.

We will distinguish three different ranges of ``overlap'' between different pairs of colourings; each range requires different tools and ideas which will be outlined in Section \ref{outlinesecondmoment}.

\section{Outline}
\label{overviewsection}
From now on, let $p \in (0,1)$ be constant and $G \sim \Gnp$. 
\subsubsection*{Independence number, first moment method and the lower bound}

The chromatic number $\chi(G)$ of a random graph $G$ is closely linked to  the independence number $\alpha(G)$, and the behaviour of the independence number of random graphs is very well understood. Recall that $b=1/(1-p)$, and let
\[\alpha_0 = 2\log_b n - 2 \log_b \log_b n +2 \log_b \parenth{e/2 }+1 = \gamma+\frac{2}{\log b}+1 .\]
For $p$ constant, Bollob\'as and Erd\H{o}s showed in 1976 (\cite{erdoscliques}, see also Chapter 11 in \cite{bollobas:randomgraphs}) that whp,
\begin{equation}
\label{independencenumber}
 \alpha(G) = \left \lfloor \alpha_0+o(1)\right \rfloor = \left \lfloor \gamma+\frac{2}{\log b}+1+ o(1)\right \rfloor,
\end{equation}
pinning down $\alpha(G)$ to at most two values whp.

In a proper colouring each colour class forms an independent set, so for any graph $G$, $\chi(G) \ge n/\alpha(G)$. For a long time, the best known lower bound for the chromatic number of dense random graphs was obtained from this simple fact. McDiarmid \cite{mcdiarmid:chromatic} sharpened this to $n/(\alpha_0-1+o(1))$ by considering the first moment of the number of independent sets of a certain size, and finally Panagiotou and Steger \cite{panagiotou2009note} used a first moment argument on the number of colourings instead to show $\chi(G) \ge \frac{n}{\gamma+o(1)}$ whp.

The \emph{first moment method} is a simple yet powerful tool, and is based on the observation that for any integer random variable $X \ge 0$, if the first moment $\E[X]$ tends to $0$, then by Markov's inequality, $\Pb(X>0)=\Pb(X\ge1)$ tends to $0$ as well. In \cite{panagiotou2009note}, $X$ is the number of all vertex partitions of $G$ which induce valid colourings (i.e., unordered colourings) with $\frac{n}{\gamma+o(1)}$ colours. Since $\E[X]\rightarrow 0$ for an appropriate choice of the $o(1)$ term in the denominator, it follows that whp no proper colouring with this number of colours exists, and the lower bound (\ref{bestknownbounds}) follows.

It turns out, however, that for $p>1-1/e^2$, the chromatic number of $\Gnp$ can not in general be found near $\frac{n}{\gamma}$. This is because for colourings with about $\frac{n}{\gamma}$ colours, the average colour class size $\gamma$ gets so close to $\alpha(G)$ that there are simply not enough disjoint independent sets of size at least $a := \left \lfloor \gamma \right \rfloor +1$. 

Note that in this case $\alpha_0-\gamma=\frac{2}{\log b}+1\in(1,2)$, so it follows from (\ref{independencenumber}) that $\alpha(G)=a$ or $\alpha(G)=a+1$ whp as shown in Figure \ref{figurecases}. In particular, there are whp no independent sets larger than $a+1$. Therefore, any colouring with average colour class size about $\gamma$ must contain a certain proportion of colour classes of size at least $a$ (and at most $a+1$).

In Section \ref{lowerboundsection}, we shall consider the number of such partial colourings with large colour classes (or rather, the number of sets of disjoint large independent sets inducing them) which are required for colourings with average colour class size of a little more than $\gamma-x_0$, where $x_0$  is the solution of (\ref{aste}). We will show that their expected number is $o(1)$, so whp no such partial colouring and hence no such complete colouring of $G$ exists.

\begin{figure}[tb]
\begin{center}
\begin{overpic}[width=0.7\textwidth]{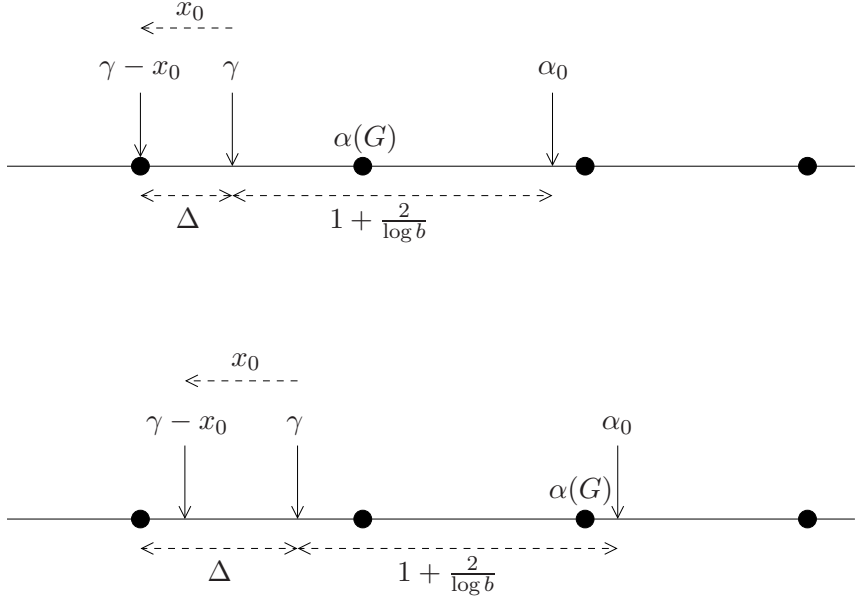}
\put(20,64){$x_0$}
\put(11,57){$\gamma-x_0$}
\put(25.5,57){$\gamma$}
\put(62.5,57){$\alpha_0$}
\put(38.5,48.5){$\alpha(G)$}
\put(38,38.8){$1+\frac{2}{\log b}$}
\put(20,38.8){$\Delta$}
\put(26.5,22.5){$x_0$}
\put(16.5,15){$\gamma-x_0$}
\put(33,15){$\gamma$}
\put(70,15){$\alpha_0$}
\put(63.8,6.8){$\alpha(G)$}
\put(46,-3){$1+\frac{2}{\log b}$}
\put(23.7,-3){$\Delta$}
\end{overpic}
\vspace{15pt}
\caption{\label{figurecases} If $p>1-1/e^2$, whp either $\alpha(G)=a=\left \lfloor \gamma \right \rfloor+1 $ (top picture) or $\alpha(G)=a+1$ (bottom picture). In the first case, there are only $o\parenth{n / \gamma}$ independent sets of size $\left \lfloor \gamma \right \rfloor+1 $, and so the colouring rate $n/\chi(G)$ drops back to the next smaller integer $\left \lfloor \gamma \right \rfloor$, i.e., $x_0=\Delta$. In the second case, there are enough independent sets of size $\left \lfloor \gamma \right \rfloor+1 $, but not necessarily enough \emph{disjoint} ones, and we have to correct the colouring rate $n/\chi(G)$ by $x_0 \in [0, 1-\frac{2}{\log b}]$ to reflect this. }
\end{center}
\end{figure}

\subsubsection*{The second moment method and the upper bound}

The upper bound in Theorem \ref{maintheorem} will be proved using the second moment method. Fix an arbitrary $\epsilon>0$. If we can show that whp,
\begin{equation}
\label{upperboundeps}
 \chi(G) \le \frac{n}{\gamma-x_0-2\epsilon},
\end{equation}
this suffices to establish the upper bound in Theorem \ref{maintheorem}. Let $\ftn=\ftn(n)\in [0,\epsilon]$ be a function, and let 
\[
k= k(n) = \left \lceil \frac{n}{\gamma - x_0-\ftn} \right \rceil.
\]
We will study $k$-colourings of $G$. For the proof of Theorem \ref{maintheorem}, we will simply pick $\ftn(n)=\epsilon$ for every $n$, but most of the second moment calculations are valid for this more general definition of $k$. To make them available for reuse in related contexts, we will work with an arbitrary function $\ftn \in [0,\epsilon]$ and then state explicitly when we only consider the special case $\ftn=\epsilon$.

 In the following, we will only consider \emph{equitable $k$-colourings} where the sizes of the colour classes differ by at most~$1$ (the method fails if we allow general colourings). We call a vertex partition into $k$ parts a \textit{$k$-equipartition} if the part sizes differ by at most~$1$. An ordered partition ist called an \textit{ordered $k$-equipartition} if the sizes of the $k$ parts differ by at most~$1$ and decrease in size (so the parts of size $\ceilnk$ come first, followed by the parts of size~$\floornk$).

Denote by $Z_k$ the \emph{number of ordered $k$-equipartitions which induce proper colourings}, i.e., where all parts form independent sets. Then our goal will be to bound the second moment of $Z_k$ in terms of $\E[Z_k]^2$. More specifically, our aim will be to show that for $n$ large enough and if $\ftn(n)=\epsilon$ for all $n$, then
\begin{equation}
\label{theaim}
 \frac{\E[Z_k^2]}{\E[Z_k]^2} \le \exp \left( \frac{n}{\log^7 n} \right).
\end{equation}
Let us briefly discuss why (\ref{theaim}) suffices to prove (\ref{upperboundeps}). By the Paley--Zygmund inequality, (\ref{theaim}) implies that
\[
\Pb \parenth{ Z_k >0} \ge  \frac{\E[ Z_k]^2}{\E[ Z_k^2]} \ge  \exp \left( -\frac{n}{\log^7 n}  \right)
\]
for large enough $n$. Therefore, for $n$ large enough, if $\ftn(n)=\epsilon$ for all $n$,
\begin{equation}
\label{final1}
 \Pb(\chi(G) \le k) \ge  \Pb \parenth{ Z_k >0}  \ge \exp \left( -\frac{n}{\log^7 n} \right).
\end{equation}
The term on the right-hand side of course tends to $0$, so it may at first seem that (\ref{final1}) is not particularly helpful in proving (\ref{upperboundeps}). However, as first noted by Frieze in \cite{frieze1990independence}, all is not lost in cases like these where we have a lower bound on a probability which tends to $0$ sufficiently slowly. Using martingale inequalities, we will see that the chromatic number of random graphs is concentrated so tightly around its mean that by adding only a few additional colours, we can boost the lower bound (\ref{final1}) to a bound which tends to $1$.

Indeed, for $G \sim \Gnp$, consider the \emph{vertex exposure martingale} (for more details see \cite{mcdiarmid1989method}, or Chapter 2.4 in \cite{janson:randomgraphs}): arbitrarily modifying the edges incident with any particular vertex of $G$ can change the value of $\chi(G)$ by at most $1$. Therefore, if we consider the martingale which is defined by the conditional expectation of $\chi(G)$ given the edges between the first $n'\le n$ vertices, it follows from the Azuma--Hoeffding or McDiarmid inequality that for all $t \ge 0$,
\begin{equation}
\label{final2}
\Pb \bigg(|\chi(G)-\E\parenth{\chi(G)}|\ge t\bigg) \le 2\exp\parenth{-\frac{t^2}{2n}}. 
\end{equation}
This implies that  $k \ge \E\parenth{\chi(G)} - \frac{n}{\log^3 n}$ for $n$ large enough and if $\ftn(n)=\epsilon$, because otherwise (\ref{final2}) with $t= \frac{n}{\log^3 n}$ would contradict (\ref{final1}). But then again by (\ref{final2}), if we let $\hat k = k+ \frac{2n}{\log^3 n}$,
\[
\Pb \parenth{ \chi(G) > \hat k} \le \Pb \parenth{ \chi(G) > \E\parenth{\chi(G)} + \frac{n}{\log^3 n}} \rightarrow 0
\]
as $n \rightarrow \infty$. Recalling that we let $\ftn(n)=\epsilon$ for all $n$, whp
\[ \chi (G) \le \hat k = k+ \frac{2n}{\log^3 n}\le\frac{n}{\gamma-x_0-2\epsilon},\]
as required.

So to prove the upper bound in Theorem \ref{maintheorem}, it remains to show (\ref{theaim}) given that $\ftn(n)=\epsilon$ for all $n$. Note that as\[
             Z_k = \sum_{\pi \text{ an ordered $k$-equipartition}} \mathbbm{1}_{\lbrace\text{$\pi$ induces a proper colouring}\rbrace},
            \]
by linearity of the expectation,
\begin{align}
 \E[Z_k^2] &= \sum _{\pi_1, \pi_2 \text{ ordered $k$-equipartitions}} \Pb \left( \text{both $\pi_1$ and $\pi_2$ induce proper colourings} \right), \label{summands}
\end{align}
where the joint  probability that both $\pi_1$ and $\pi_2$ induce proper colourings of course depends critically on how similar they are.

Classifying the amount of overlap between $\pi_1$ and $\pi_2$ and splitting up the calculation into manageable cases will be the main challenge of the proof.  In Section \ref{thesecondmomentsection}, we will first quantify the amount of overlap between two partitions, and in Sections \ref{sectiontypicalrange}--\ref{upper}, we will proceed to distinguish three different ranges of overlap and bound their respective contributions to (\ref{theaim}). Each range will be tackled through a different approach. A more detailed overview of the different ideas  for each range is given in Section \ref{outlinesecondmoment}.

\subsection*{Remark}

Like Bollob\'as' original proof of the asymptotic upper bound \cite{bollobas1988chromatic}, our proof requires the use of martingale concentration inequalities. This is necessary because for our choice of $k$, $\E[Z_k^2]/\E[Z_k]^2 \nrightarrow 1$, so the second moment method alone cannot yield the whp existence of a colouring.

However, it is possible to obtain the upper bound $\left \lceil \frac{n}{\left \lfloor \gamma \right \rfloor -1} \right \rceil= \frac{n}{\gamma-\Delta-1+o(1)}$ using only the second moment method. For this, we would need to work in $\Gnm$ with $m \approx p {n \choose 2}$ instead of $\Gnp$ and only consider colourings where all colour classes are of size exactly $\left \lfloor \gamma \right \rfloor-1$ (increasing $n$ slightly if $\left \lfloor \gamma \right \rfloor-1$ does not divide $n$).

Working with colour classes of size $\left \lfloor \gamma \right \rfloor -1$ would also simplify the calculations considerably, as much of the technical difficulty in our proof comes from colour classes of size at least $\left \lfloor \gamma \right \rfloor$ which do not exist in this setting.

\section{Preliminaries and notation}
\label{preliminariessection}
From now on, we will always assume that $n$ is large enough so that various bounds and approximations hold, even when this is not stated explicitly.

 For two functions $f=f(n)$, $g=g(n)$, we say that $f$ is asymptotically at most $g$, denoted by $f \lesssim g$, if $f(n) \le (1+o(1))g(n)$ as $n \rightarrow \infty$. Analogously, $f \gtrsim g$ means that $f(n) \ge (1+o(1))g(n)$. We write $f = O(g)$ if there are constants $C$ and $n_0$ such that $|f(n)| \le Cg(n)$ for all $n\ge n_0$. Furthermore, we say that $f = \Theta(g)$ if $f = O(g)$ and $g= O(f)$.

Recall that $\gamma= 2 \log_b n -2 \log_b \log_b n-2 \log_b 2$, and that we fix an arbitrary $\epsilon>0$ and a function $\ftn=\ftn(n) \in [0,\epsilon]$, and let 
\begin{equation} k= \left \lceil \frac{n}{\gamma -x_0-\ftn} \right \rceil      \text{ and } \, l=  \left \lfloor \frac{n}{\gamma -x_0+\epsilon} \right \rfloor. \label{defofkl} \end{equation}

Later on, we will simply pick $\ftn(n)=\epsilon$ for all $n$. As described in the outline, we are going to show that for any constant $\epsilon>0$ and if $\theta(n)=\epsilon$ for all $n$, then whp $\chi(G)\ge l$, and $\Pb \parenth{\chi(G) \le k} \ge \exp \parenth{-\frac{n}{\log^7 n}}$. Most of the calculations are valid for the more general definition of $k$, and so we will work in this more general context unless $\ftn=\epsilon$ is specified explicitly.
Let 
\[\delta= \frac{n}{k}-\left\lfloor \frac{n}{k} \right\rfloor ,\,\,\,\,\,\,\,\,\,\,\,\, k_1 = \delta k \,\,\,\,\,\, \text{ and } \,\,\,\,\,\, k_2=(1-\delta)k.\]
If $k$ does not divide $n$, then a \emph{$k$-equipartition} consists of exactly $k_1$ parts of size $\left\lceil \frac{n}{k} \right\rceil$ and exactly $k_2$ parts of size $\left\lfloor \frac{n}{k} \right\rfloor$. In an \emph{ordered $k$-equipartition}, the first $k_1$ parts are of size $\left\lceil \frac{n}{k} \right\rceil$ and the remaining $k_2$ parts are of size $\left\lfloor \frac{n}{k} \right\rfloor$.

Let $P$ denote the {total number of ordered $k$-equipartitions} of the $n$ vertices, then
\begin{equation}
\label{Pdefinition}
P = \frac{n!}{\left\lceil\frac{n}{k}\right\rceil!^{k_1}\left\lfloor\frac{n}{k}\right\rfloor!^{k_2}}.
\end{equation}
Since by Stirling's approximation, $n! = \Theta \parenth{n^{n+1/2}e^{-n}}$,
\begin{equation}
\label{pasymp}
P= k^n \exp(o(n)).
\end{equation}
Given a $k$-equipartition, there are exactly
\begin{align}
f &= k_1  {\lceil n/k \rceil \choose 2}+k_2  {\lfloor n/k \rfloor \choose 2} 
= \frac{n\left(\frac{n}{k}-1\right)}{2}+\frac{\delta(1-\delta)}{2} k \label{defoff}
\end{align}
\textit{forbidden edges} which are not present in $G$ if the partition induces a proper colouring. Therefore, the probability that a given ordered $k$-equipartition induces a proper colouring is exactly $q^f $, so
\begin{align}
\label{firstmoment}
\mu_k :=\E[Z_k] =  P q^{f}.
\end{align}

Note that 
\[\frac{n}{k} \le \gamma-x_0 - \ftn \le \gamma,\]
so $\ceilnk \le \left \lceil \gamma \right \rceil \le \left \lfloor \gamma \right \rfloor +1 = a$. As $x_0 \le \Delta$, where $\Delta= \gamma - \left \lfloor \gamma \right \rfloor$ (see Section \ref{technicallemmas}) and $\ftn\le \Delta$,
\[
\frac{n}{k} = \gamma-x_0 - \ftn+o(1) \ge \left \lfloor \gamma \right \rfloor-\epsilon+o(1)=a-1-\epsilon+o(1), 
\]
so for $n$ large enough, 
\[a-3-\epsilon \le \floornk \le \ceilnk \le a.\]

\subsection{List of key facts and relations}
\label{sectionkeyfacts}
Below is a list of some facts, bounds and approximations so that we can conveniently refer back to them later on.
\begin{enumerate}[(A)]
\item \label{keyalphaa} If $p>1-1/e^2$, then whp $\alpha(G) \in \lbrace a, a+1 \rbrace$, where $\alpha(G)$ denotes the independence number and $a = \left \lfloor \gamma \right \rfloor+1$ (see Section \ref{overviewsection}).
\item For $n$ large enough, $a-3-\epsilon \le \floornk \le \ceilnk \le a$, where $a = \left \lfloor \gamma \right \rfloor+1$. 
\label{keyceilnk}
\item In a $k$-equipartition, there are $k_1=\delta k$ parts of size $\ceilnk$ and $k_2= (1-\delta)k$ parts of size $\floornk$, where $\delta=\frac{n}{k}-\floornk$. \label{keyequipartition}
\item $\gamma \sim a \sim \frac{n}{k}  \sim \frac{n}{l} \sim 2 \log_b n = \Theta \parenth{\log n}$ \label{key3}
\item $k \sim l \sim \frac{n}{\gamma} \sim \frac{n}{a} \sim \frac{n}{2 \log_b n}  = \Theta \parenth{\frac{n}{\log n}}$ \label{key1}
\item $q^{-\gamma/2} = b^{\gamma/2} = \frac{n}{2 \log_b n}  \sim k \sim l$ \label{key2}

\item $k^{\frac{1}{n/k}} = O(1)$ and $k^{\frac{1}{\log n}} = O(1)$. \label{key7}
\item \label{key6} 
$
 f \sim n \log_b n.$
\addtocounter{enumi}{1}\item \label{keychoose} For any integer function $\varphi=\varphi(n) =o(n)$, $ {n \choose \varphi} \le \exp \parenth{ o(n)}$.
\item \label{exponentialexpectation}
$\frac{\mu_k}{k_1! k_2!} \ge b^{\ftn n /2} \exp(o(n)).$
A proof is given in the appendix.

\end{enumerate}

\subsection{On the solutions of (\ref{aste})}
\label{technicallemmas}

In this section we will explore the solutions of the  inequality (\ref{aste}) in Theorem \ref{maintheorem} and state some technical lemmas. Let
\[
\varphi(x) = \varphi_n(x) = (1-\Delta+x) \log_b(1-\Delta+x) +(1-\Delta)(\Delta-x)/2.       
       \]
Then $x_0$ is defined in Theorem \ref{maintheorem} as the smallest nonnegative solution of $\varphi(x) \le 0$. Since $\varphi(\Delta)=0$, $x_0$ is well-defined and $x_0 \in [0, \Delta]$. Note that
\begin{align*}\varphi'(x) &= \log_b (1-\Delta+x)+\frac{1}{\log b}-\frac{1-\Delta}{2}\\
\varphi''(x) &= \frac{1}{(1-\Delta+x)\log b} \ge 0.
\end{align*}
Therefore, $\varphi$ is convex and there are three different possible cases for the location of $x_0$ as shown in Figure \ref{figuresolutions}. In the first case, $\varphi(0) \le 0$ and therefore $x_0=0$. In the second and third case, $\varphi(0) > 0$, so $x_0>0$. In the second case, $x_0$ lies strictly between $0$ and $\Delta$, and in the third case, $x_0=\Delta$, which happens if and only if $\varphi'(\Delta)\le 0$, or equivalently $1-\Delta \ge \frac{2}{\log b}$. This case corresponds to the upper picture in Figure \ref{figurecases}.

The following two lemmas are needed to obtain Corollary \ref{maincorollary} from Theorem \ref{maintheorem} and will be proved in the appendix.
\begin{lemma}
\label{corollary0}
If $p \le 1-1/e^2$, then $x_0=0$.
\end{lemma}
\begin{lemma}
\label{corollarybounds}
 If $p>1-1/e^2$, then $0\le x_0 \le 1-\frac{2}{\log b}$.
\end{lemma}
The proofs of the following technical lemmas are straightforward analytical arguments and can also be found in the appendix.
\begin{figure}[bt]
\begin{center}
\begin{overpic}[width=0.95\textwidth]{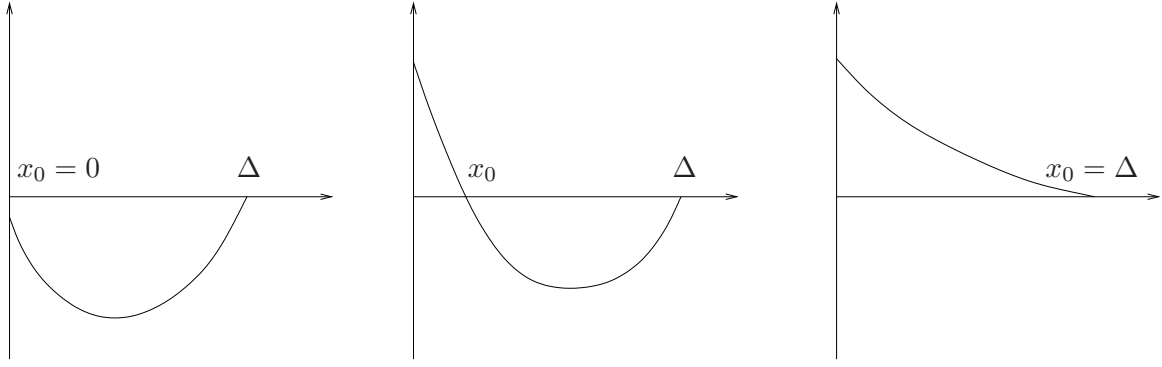}
\put(1,16){$x_0=0$}
\put(20,16){$\Delta$}
\put(40,16){$x_0$}
\put(57.7,16){$\Delta$}
\put(90,16){$x_0=\Delta$}
\end{overpic}
\vspace{15pt}
\caption{\label{figuresolutions} The three possible cases for the function $\varphi(x)$. In the first case, $\varphi(0) \le 0$, so $x_0=0$. In the second case, $x_0 \in (0, \Delta)$. In the third case, $\Delta$ is the smallest nonnegative solution of $\varphi(x)\le 0$, so $x_0=\Delta$.}
\end{center}
\end{figure}

\begin{lemma}
 \label{technicallemma1}
Suppose $p>1-1/e^2$, and fix $\epsilon'>0$. Then there is a constant $c_1=c_1(\epsilon')>0$ such that if $x_0>\epsilon'$, then 
\begin{equation*}
\varphi(x_0-\epsilon') \ge c_1. \label{constantlowerbound}
\end{equation*}
\end{lemma}

\begin{lemma}
\label{technicallemma2}There is a constant $c_2=c_2(\epsilon) \in (0,1)$ such that if $x_0 \le \Delta-\epsilon$, then 
\[
 1-\Delta \le \frac{2c_2}{\log b}.
\]
\end{lemma}

\begin{lemma}
 \label{technicallemma3}
Fix $\epsilon'>0$. There is a constant $c_3=c_3(\epsilon, \epsilon')>0$ such that if $\epsilon'\le y\le \Delta-x_0-\epsilon$, then
\[
\varphi(\Delta-y) \le -c_3.
\]
\end{lemma}

\begin{lemma}
 \label{technicallemma4}
Fix $\epsilon'>0$. There is a constant $c_4=c_4(\epsilon, \epsilon')>0$ such that if  $\epsilon' \le \Delta-x_0-\epsilon \le y\le1$, then
\[
(1-y) \log_b (1-y)+\frac{\Delta}{2}(1-y)-\frac{x_0+\epsilon}{2} \le -c_4.
\]
\end{lemma}

\section{Proof of the lower bound}
\label{lowerboundsection}

We may assume that $p > 1-1/e^2$, because otherwise $x_0=0$ and the lower bound in Theorem \ref{maintheorem} is simply the known lower bound (\ref{bestknownbounds}). Recall that we let $l= \left\lfloor\frac{n}{\gamma-x_0+\epsilon} \right\rfloor$ for an arbitrary fixed $\epsilon>0$. We may assume $x_0-\epsilon \ge 0$, because we can just use the known bound (\ref{bestknownbounds}) instead for all $n$ where this is not the case. We will show that any $l$-colouring must contain a certain proportion of large colour classes of size $a= \left \lfloor \gamma \right \rfloor+1$, and then prove that the expected number of unordered partial colourings with just these large colour classes tends to $0$, which means that whp no such partial and therefore no complete $l$-colouring exists.

As $x_0 \le \Delta= \gamma-\left \lfloor \gamma \right \rfloor$ (see Section \ref{technicallemmas}) and since $x_0-\epsilon \ge 0$, 
\[\left\lfloor \gamma \right\rfloor < \gamma-x_0+\epsilon \le \gamma,\]
so $\left \lceil\gamma-x_0+\epsilon\right \rceil = \left\lfloor \gamma \right\rfloor+1 =a$. This is very close to the independence number $\alpha(G)$: by (\ref{keyalphaa}) from Section $\ref{sectionkeyfacts}$, whp $\alpha(G) =a$ or $\alpha(G)=a+1$. In particular, whp there are no independent sets of size $a+2$ in $G$.

Recall that $\alpha_0= \gamma+1+\frac{2}{\log b}$. Standard calculations show that for any $t=t(n)=O(1)$ such that $\alpha_0-t$ is an integer, the expected number of independent sets of size $\alpha_0-t$ in $G$ is $n^{t+o(1)}$ (see also 3.c) in \cite{mcdiarmid1989method}). Therefore, since $a=\left \lfloor \gamma \right \rfloor +1 = \gamma-\Delta+1= \alpha_0-\frac{2}{\log b}-\Delta$, the expected number of independent sets of size $a$ and $a+1$, respectively, can be calculated as
\begin{align}
{n \choose a} q^{a \choose 2} &= n^{\frac{2}{\log b}+\Delta+o(1)} \text{, and} \nonumber \\
{n \choose a+1} q^{a+1 \choose 2} &= n^{\frac{2}{\log b}+\Delta-1+o(1)} 
 \label{expectationfora}.
\end{align}
Note that as $\frac{2}{\log b}+\Delta-1 < \frac{2}{\log b}<1$ and by (\ref{key1}), $l= \Theta\parenth{\frac{n}{\log n}}$, it follows from Markov's inequality that whp only $o(l)$ independent sets of size $a+1$ exist in $G$.

We assume from now on that no independent sets of size $a+2$ and only $o(l)$ independent sets of size $a+1$ are present in $G$, both of which hold whp. Under this assumption, a valid $l$-colouring must contain a certain proportion of colour classes of size exactly $a$, since the average colour class size is $\frac{n}{l}\ge\gamma-x_0+\epsilon$. More specifically, given an $l$-colouring, if we let $y\in [0,1]$ be the proportion of colour classes of size $a$ in the colouring, and let $z=o(1)$ be such that there are exactly $zl$ independent sets of size $a+1$ in $G$, then adding up the number of vertices in each colour class yields
\[
n \le ayl+(a+1)zl+(a-1)(1-y-z)l.
\]
Therefore, since $n/l \ge  \gamma-x_0+\epsilon$,
\[
 \gamma-x_0+\epsilon \le y+a-1+2z.
\]
As $a= \gamma-\Delta+1$ and $z=o(1)$, it follows that
\[
y \ge \Delta-x_0+\epsilon+o(1).
\]
Hence, as $l \sim \frac{n}{2\log_b n}$  by (\ref{key1}), if a proper $l$-colouring exists and $n$ is large enough, then in particular $G$ must contain at least
\[
s := \left \lceil \frac{\parenth{ \Delta-x_0+\epsilon/2} n}{2 \log_b n} \right \rceil 
\]
disjoint independent sets of size $a$. We shall call such an (unordered) collection of $s$ disjoint independent sets of size $a$ a \emph{precolouring}, and denote by $\bar{Z}$ the number of precolourings in $G$.

Since for all $m \in \N$, $m^m/e^m \le m! \le m^{m+O(1)}/e^m$,
\begin{align*}
\E[\bar{Z}] &= \frac{1}{s!}{n \choose a}{{n-a} \choose a} \cdots {{n-(s-1)a} \choose a}q^{s  {a \choose 2}}= \frac{n!q^{s  {a \choose 2}} }{s!a!^s(n-as)!}  \le \frac{e^{s-as}n^{n+O(1)}q^{s  {a \choose 2}}}{s^sa!^s (n-as)^{n-as} }\\
&= n^{O(1)} \parenth{\frac{n^a  q ^{{a \choose 2}}}{e^{a-1}sa!}}^s \parenth{\frac{n}{n-as}}^{n-as}.
\end{align*}
By (\ref{key1}), $a\sim 2\log_b n$, so it follows from (\ref{expectationfora}) that $\frac{n^a}{a!} q ^{{a \choose 2}} \sim {n \choose a}q ^{{a \choose 2}} = n^{\frac{2}{\log b}+\Delta+o(1)}$. Furthermore, $e^{a-1}=n^{\frac{2}{\log b} +o(1)}$. Therefore, since $n^{1-o(1)}\le s \le n$,
\begin{align*}
\E[\bar{Z}] &\le n^{O(1)}\parenth{n^{\Delta+o(1)}s^{-1}}^s \parenth{1-\frac{as}{n}}^{-n\parenth{1-\frac{as}{n}}}\le e^{o(n)}n^{-s(1-\Delta)}  \parenth{1-\frac{as}{n}}^{-n\parenth{1-\frac{as}{n}}}.
\end{align*}
As $as/n \sim \Delta-x_0+\epsilon/2$, this gives
\begin{align*}
\E[\bar{Z}] &\le e^{o(n)} n^{-s(1-\Delta)} \parenth{1-\Delta+x_0-\frac{\epsilon}{2}}^{-\parenth{1-\Delta+x_0-\frac{\epsilon}{2}}n}\\
&=b^{-\parenth{\parenth{1-\Delta+x_0-\frac{\epsilon}{2}} \log_b \parenth{1-\Delta+x_0-\frac{\epsilon}{2}}+(1-\Delta)\parenth{\Delta-x_0+\frac{\epsilon}{2}}/2+o(1)}n}.
\end{align*}
Note that with the exception of the $o(1)$ term, the expression in the exponent is now simply the left-hand side of condition (\ref{aste}) in Theorem \ref{maintheorem} with $x= x_0-\epsilon/2$. As $\epsilon/2<\epsilon \le x_0$, we may apply Lemma \ref{technicallemma1} with $\epsilon'=\epsilon/2$ to conclude that
\[
 \E[\bar{Z}] \le b^{-(c_1+o(1))n}=o(1).
\]
By Markov's inequality, whp no precolouring and consequently no proper $l$-colouring exists. \qed

\section{Bounding the second moment}
\label{thesecondmomentsection}
Recall that to prove Theorem \ref{maintheorem}, it remains to show that for an arbitrary fixed $\epsilon>0$, if we let $\theta(n)= \epsilon$ for all $n$ in the definition (\ref{defofkl}) of $k$, then if $n$ is large enough,
\[
 \E[Z_k^2] / \mu_k^2 \le \exp \parenth{\frac{n}{\log^7 n}}.
\]

By (\ref{summands}), in order to bound $\E[Z_k^2]$, we need to study the joint probability that two partitions both induce proper colourings, a quantity which of course depends on how similar the two partitions are. To quantify the amount of overlap between two partitions, we define the \textit{overlap sequence} $\mathbf{r}$. For $2 \le i \le a= \left \lfloor \gamma \right \rfloor+1$, given two ordered $k$-equipartitions $\pi_1$, $\pi_2$, denote by $ r_i $ the number of pairs of parts (the first being a part in $\pi_1$ and the second being a part in $\pi_2$) which intersect in exactly $i$ vertices. Denote by
\[\bfr=(r_2, r_3, \dots, r_a)\]
the \emph{overlap sequence} of the two ordered $k$-equipartitions $\pi_1$, $\pi_2$. If the intersection of two parts contains at least two vertices, we call the intersection an \textit{overlap block}. If there is only a single vertex in the intersection of two parts, we call that vertex a \emph{singleton}. Note that since by (\ref{keyceilnk}) from Section \ref{sectionkeyfacts}, $\ceilnk \le a$ for $n$ large enough, no overlap block is larger than $a$.

Conversely, given an overlap sequence $\mathbf{r}$, denote by $P_{\mathbf{r}}$ the number of \textit{ordered pairs} of ordered $k$-equipartitions with overlap sequence $\mathbf{r}$. Let 
\[v= v(\bfr) =\sum_{i=2}^{a} i r_i\]
 be the \textit{number of vertices involved in the overlap}, and let
\[\rho = v/n \le1\]
denote the \textit{proportion} of those vertices in the graph. Furthermore, denote by
\begin{equation}
\label{defofd}
d = d(\mathbf{r}) = \sum_{i=2}^{a} r_i {i \choose 2}                                                                                                                                                                                                             \end{equation}
the number of \textit{common forbidden edges} that two ordered $k$-equipartitions  $\pi_1$, $\pi_2$ with overlap sequence $\mathbf{r}$ share. Since the number of forbidden edges in one partition is exactly $f$, where $f$ was defined in (\ref{defoff}), both  $\pi_1$ and  $\pi_2$ induce proper colourings if and only if none of the exactly $2f-d$ forbidden edges are present. Therefore, from (\ref{summands}) and (\ref{firstmoment}),
\begin{align*}
\E[Z_k^2] &= \sum _{\mathbf{r}} P_{\mathbf{r}}q^{2f-d} = \mu_k^2 \sum _{\mathbf{r}} \frac{P_{\mathbf{r}}}{P^2}  b^d.
\end{align*}
Let
\begin{align}
 Q_\bfr &=  \frac{P_\bfr}{P^2}, \label{defofq} 
\end{align}
then our goal is to show that for $n$ large enough, if $\ftn(n)=\epsilon$ for all $n$,
\begin{equation}
\label{thesum}
\frac{\E[Z_k^2] }{\mu_k^2} = \sum_\bfr Q_\bfr  b^{d} \le \exp \parenth{\frac{n}{\log^7 n}}.
\end{equation}
Since the summands in (\ref{thesum}) vary considerably for different types of overlap sequences $\bfr$, we split up our calculations into three parts in Sections \ref{sectiontypicalrange} -- \ref{upper}. The behaviour of the summands is rather different in each case, and so different methods will be required to bound them.

\subsection{Outline}
\label{outlinesecondmoment}
\subsubsection*{Typical case}
In Section \ref{sectiontypicalrange}, we first discuss the typical form of overlap between pairs of partitions. If a partition is chosen uniformly at random from all possible ordered $k$-equipartitions, then the probability that two given vertices are in the same part is roughly $\frac{1}{k}$. Consequently, if two ordered $k$-equipartitions are sampled independently and uniformly at random, then the expected number $d$ of forbidden edges they have in common, i.e., pairs of vertices which are in the same part in both partitions, is of order $\frac{n^2}{k^2}=O(\log^2 n)$. In particular, we do not expect the number $v$ of vertices involved in the overlap to be much larger than $2d=O(\log^2 n)$.

Furthermore, the expected number of \emph{triangles} which the two partitions have in common, i.e., triples of vertices that are in the same part in both partitions, is of order $\frac{n^3}{k^4}=O\parenth{\frac{\log^4 n}{n}}=o(1)$. Therefore, typically two partitions have no triangles or larger cliques in common, and overlap in about $O(\log^2 n)$ disjoint pairs of vertices.

In fact, we will cover a much larger range of overlap sequences $\bfr$ in Section \ref{sectiontypicalrange}, namely those $\bfr$ where at most a constant fraction of all vertices are involved in the overlap, i.e., where $v=v(\bfr)\le cn$ for a constant $c$ which will be defined in (\ref{defofc}).

To bound the number of such pairs of partitions, we will count the number of corresponding \emph{overlap matrices}. The overlap matrix between two partitions $\pi_1$ and $\pi_2$ is defined as the matrix $\mc{M}=\left( M_{xy} \right)$, where $M_{xy}$ denotes the number of vertices that are in part number $x$ in $\pi_1$ and in part number $y$ in $\pi_2$. If $\pi_1$ and $\pi_2$ overlap according to  a given overlap sequence $\bfr$, then the entries of $\mc{M}$ are exactly $r_i$ instances of the number $i$ for all $2 \le i \le \ceilnk$, as well as $n-v$ instances of the number $1$, with the remaining entries $0$. As $\pi_1$ and $\pi_2$ are ordered $k$-equipartitions, all rows and columns of $\mc{M}$ sum to $\ceilnk$ or $\floornk$.

Since there are typically few pairs and very few triangles or larger cliques in the overlap, one crucial idea is that we can count the number of overlap matrices by first placing any entries $2, 3, \dots, \ceilnk$ in the matrix separately, and then treating the rest of the matrix as a $0-1$ matrix with given row and column sums close to $\frac{n}{k}$. An important tool is Theorem \ref{mckaymatrices}, due to McKay, which gives an estimate for the number of $0-1$ matrices with prescribed row and column sums.

After some fairly accurate calculations, we will see that the contribution from each $\bfr$ in this case is bounded by an expression of the form $\prod_{i=2}^a \frac{T_i^{r_i}}{r_i!}$, where the terms $T_i$ still depend on $\rho(\bfr)=v/n$. We will then show that if $\rho\le c$, the terms $T_i$ are small enough so that the overall contribution to (\ref{thesum}) is bounded by $\exp \parenth{\frac{n}{\log^8 n}}$. The bound for the term $T_\ceilnk$ will require condition (\ref{aste}) from Theorem \ref{maintheorem} to hold.

Let us remark that if we work with $\mathcal{G}(n,m)$ instead of $\Gnp$ and conduct a much more detailed analysis, it is possible to show that if $p<1-1/e$,
 the contribution from this range of $\bfr$ is bounded by a constant. The bulk of this contribution comes from overlap sequences of the form $\bfr=(r_2, 0, 0,\dots, 0)$ with $r_2=O(\log^2 n)$. However, only the coarser bound is needed for our result.

\subsubsection*{Many small overlap blocks}

An intermediate degree of overlap is examined in Section \ref{middle}, where at least a constant fraction $cn$ of vertices are involved in the overlap between the two partitions, but there are either still many small overlap blocks, or many vertices not involved in the overlap at all. More specifically, for an arbitrary constant $c'>0$, we will consider all  $\bfr$ with $\rho=v/n>c$ and
\[
\sum_{2 \le i \le 0.6\gamma} i r_i \ge {c'} n  \,\text{ or }\, \rho \le 1-{c'},
\]
i.e., those $\bfr$ where there are either at least $c'n$ vertices not in the overlap, or at least $c'n$ vertices in overlap blocks of size at most $0.6 \gamma$.

Let us assume for the moment that $\frac{n}{k}$ is an integer in order to simplify notation. It will be useful to define a simple density parameter $\beta$ which measures how close the overlap of two ordered $k$-equipartitions is to consisting entirely of complete parts of size $\frac{n}{k}$ (with the remaining $n-v$ vertices being singletons not involved in the overlap).

Any overlap block contains at most $\frac{n}{k}$ vertices. Therefore, if we view the overlap blocks of two $k$-equipartitions as cliques making up a graph, then each vertex has degree at most $\frac{n}{k}-1$ within this overlap graph. Hence, given the number $v$ of vertices involved in the overlap and the number $d$ of common forbidden edges, we know that $2d \le (\frac{n}{k}-1)v$, and we let
\[
 \beta = \frac{2d}{(\frac{n}{k}-1)v} \le 1.
\]
If $\beta$ is close to $1$, then the overlap consists almost entirely of very large overlap blocks which are almost entire parts.

In Section \ref{contribution12}, we will first consider the case where $\beta$ is not too close to $1$ (so there are enough small overlap blocks). Thereafter, in Section \ref{contribution3}, we will study the case where $\beta$ is close to $1$ (so the overlap of the pairs of partitions consists almost exclusively of very large overlap blocks), but there are still many vertices which are not involved in the overlap at all, i.e., $n-v$ is large enough.

In both cases, we will bound the number $P_\bfr$ of pairs of ordered $k$-equipartitions with overlap sequence $\bfr$ according to the same strategy. We fix the first ordered $k$-partition $\pi_1$ arbitrarily, and then we generate the second partition in the following way. 

We first subdivide the parts of $\pi_1$ into overlap parameter blocks and singletons according to $\bfr$. In the first case, a fairly slack bound on the number of ways to do this will suffice (Lemma \ref{choppinguplemma1}). In the second case, we need to be more careful, and so we will show that this can be done in subexponentially many ways (Lemma \ref{choppinguplemma3}).

Thereafter, we sort the overlap blocks and singletons into $k$ parts in order to form the new partition $\pi_2$. In the first case the number of ways to do this is simply bounded by $k^{R+n-v}$, where $R$ denotes the number of overlap blocks. Bounding $R$ in terms of $\beta$ in (\ref{eqforr1}) and (\ref{eqforr2}), we will see that the overall contribution from the first case to the sum (\ref{thesum}) is $o(1)$.

In the second case we again need a better bound for the number of ways to sort the overlap blocks into the $k$ parts in order to form $\pi_2$. Note that in this case, almost the entire overlap consists of very large overlap blocks. If we sort these large overlap blocks into the $k$ parts first, then they occupy their assigned parts almost completely. As there are $v=\rho n$ vertices in the overlap, this means that roughly $\rho k$ of the $k$ parts are now filled or almost filled. The remaining smaller overlap blocks and singletons, of which there are roughly $(1-\rho)n$, do not have $k$ parts to pick from. Instead, their choice is limited to about $(1-\rho)k$ parts. Therefore, in this case we get an additional factor of roughly $(1-\rho)^{(1-\rho)n}$.

As almost everything else turns out to be subexponential in the second case, this would be the end of the story if $\frac{n}{k}$ were indeed an integer: as long as $c<\rho<1-c'$, the overall contribution to the sum (\ref{thesum}) would decrease exponentially, and in particular it would be $o(1)$.

However, the fact that $\frac{n}{k}$ is not in general an integer is not purely a notational inconvenience. When we do distinguish between parts of size $\ceilnk$ and $\floornk$ (or rather, for technical reasons, between parts of size $a=\left\lfloor \gamma \right \rfloor+1$ and of size at most $a-1$) then there is an additional factor of size about $b^{v_1(1-\Delta)/2}$, where $v_1$ denotes the number of vertices in the overlap which are in parts of size $a$ within the first partition $\pi_1$.

As $v_1 \le v = \rho n$, this means that overall, in equation (\ref{yahoo}), we arrive at an expression which is roughly 
\[(1-\rho)^{(1-\rho)n}b^{\rho (1-\Delta)n/2}= b^{n ((1-\rho)\log_b(1-\rho)+\rho(1-\Delta)/2)}.\]
Noting in (\ref{equationnk}) that the proportion of vertices in sets of size $a$ in a $k$-equipartition is roughly $\Delta-x_0-\epsilon$, it is now not very hard, but slightly tedious, to compare this last exponent to condition (\ref{aste}) in Theorem \ref{maintheorem} in order to show that this expression is exponentially decreasing in $n$. We will need to consider several cases, and we will also use the technical Lemmas \ref{technicallemma3} and \ref{technicallemma4} from Section \ref{technicallemmas}.

Overall, we will show that the contribution from the second case to the sum (\ref{thesum}) is~$o(1)$.

\subsubsection*{High overlap}

Finally, in Section \ref{upper} we will study those $\bfr$ where the corresponding pairs of partitions are very similar to each other. In this range, most of the overlap consists of almost entire parts which are merely permuted, with a few exceptional small overlap blocks and singletons.

We will show that the contribution to (\ref{thesum}) from this range of overlap is $O(\frac{k_1! k_2!}{\mu_k} {k \choose k_1})$. Since we are sufficiently far above the first moment threshold for the number of colourings, this is $o(1)$, and summing up the contributions from each of the three cases yields (\ref{thesum}) and thereby concludes the proof of Theorem \ref{maintheorem}.

\begin{figure}
  \begin{center}
    \begin{overpic}[width=0.5\textwidth]{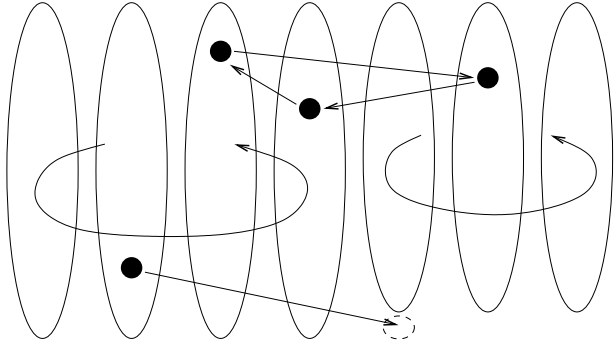}
         \end{overpic}
  \end{center}
  \caption{\label{figurepart3}In the high overlap case, the second partition is largely generated by permuting the exceptional vertices and then permuting the parts of size $\ceilnk$ (shown on the left) and of size $\floornk$ (shown on the right). Exceptional vertices may also jump to smaller parts of size $\floornk$.}
\end{figure}

It is helpful to first consider the \emph{extreme case} of those pairs of partitions $\pi_1$, $\pi_2$ which are simply permutations of each other: as there are $k_1$ parts of size $\ceilnk$ and $k_2$ parts of size $\floornk$, there are exactly $P k_1!k_2!$ such (ordered) pairs of partitions, where $P$ is defined in (\ref{Pdefinition}) as the total number of $k$-equipartitions. The number of overlapping edges is maximal, so $d=f$. Therefore, from (\ref{firstmoment}), the overall contribution to (\ref{thesum}) is exactly
\[
 \frac{Pk_1!k_2!}{P^2} b^f = \frac{k_1!k_2!}{Pq^f} = \frac{k_1!k_2!}{\mu_k}.
\]
More generally, we will consider pairs of partitions which are largely just permutations of each other, but where there are also a few \emph{exceptional vertices} which are essentially permuted amongst themselves first, as shown in Figure \ref{figurepart3}. As the part sizes may vary by $1$, however, the number of `available slots' for exceptional vertices in each of the $k$ parts may vary by $1$. We will bound these variations with the factor  ${k \choose k_1}$.

From Section \ref{middle}, we can assume that there are at most $2c'n$ exceptional vertices, where we can make the constant $c'$ as small as we like. We will distinguish three different types of exceptional vertices. Starting with the first partition $\pi_1$, we will first select the exceptional vertices of each type, and bound the number of choices in Lemma \ref{Part3Lemma1}. Then we generate $\pi_2$ and bound the number of ways to do this in Lemma \ref{Part3Lemma2}. Finally, we examine how much each exceptional vertex subtracts from the maximum number $f$ of shared forbidden edges between $\pi_1$ and $\pi_2$ in Lemma \ref{Part3Lemma3}. Summing over the number of exceptional vertices, we will see that the overall contribution to (\ref{thesum}) is of order $O(\frac{k_1! k_2!}{\mu_k}  {k \choose k_1})$ if $c'$ is chosen small enough.

\subsection{Typical overlap range}
\label{sectiontypicalrange}
We will first consider all those overlap sequences $\bfr$ where the proportion $\rho=v/n$ of the vertices which are involved in the overlap is at most 
\begin{equation}
 c= \frac{1-c_2}{2} \in \left(0,\frac{1}{2}\right), \label{defofc}
\end{equation}
where $c_2$ is the constant from Lemma \ref{technicallemma2}. So let
\[
\mc{R}_1 = \left\lbrace \bfr \mid \rho =\rho(\bfr) \le c \right\rbrace.
\]
The vast majority of all pairs of partitions overlap in a parameter sequence $\bfr \in \mc{R}_1$, and this is also where the bulk of the sum (\ref{thesum}) comes from. We will show that if $\ftn(n)=\epsilon$ for all $n$, then the contribution of the overlap sequences $\bfr\in \mc{R}_1$ to (\ref{thesum}) is at most $\exp \parenth{\frac{n}{\log^8 n}}$. To do this, we will find a bound for the contribution from each $\bfr$ of the form $\prod_i\frac{T_i^{r_i}}{r_i!}$, and then bound the terms $T_i$. We state and prove the following lemma for general functions $\ftn(n)\in[0,\epsilon]$ and for general $\bfr$ and make some simplifications for the required case $\ftn(n)= \epsilon$ and $\bfr \in \mc{R}_1$ afterwards.

\begin{lemma}
 \label{quotablelemma1}
Fix $\epsilon>0$, let $\ftn=\ftn(n)\in[0,\epsilon]$ be an arbitrary function and $k= \left \lceil \frac{n}{\gamma-x_0-\ftn}\right \rceil$. 
Recall the definitions (\ref{defofd}) and (\ref{defofq}) of $d$ and $Q_{\bfr}$. Then
\[
Q_{\bfr} b^d \lesssim 
 \prod_{i=2}^a \left(\frac{1}{r_i!} \left(\frac{e^{\rho i}b^{i \choose 2} k^{2} \ceilnk!^{2}}{ n^i i!\left( \ceilnk-i\right)!^{2}}\right)^{r_i} \right)  \exp \left( -\frac{1}{2} \left(\frac{n-v}{k} -1 \right)^2 \right).
\]
 \end{lemma} 
\begin{proof}

Any pair of ordered $k$-equipartitions $\pi_1$ and $\pi_2$ defines a $k\times k$ \textit{overlap matrix} $\mc{M}=\left( M_{xy} \right)$, where $M_{xy}$ is the number of vertices that are in part $x$ in $\pi_1$ and in part $y$ in $\pi_2$. Since $\pi_1$ and $\pi_2$ are ordered $k$-equipartitions, the first $k_1$ rows and columns of $\mc{M}$ sum to $\ceilnk$ and the remaining $k_2$ rows and columns sum to $\floornk$. If $\pi_1$ and $\pi_2$ overlap according to the overlap sequence $\bfr$, this means that for every $2\le i \le a$, exactly $r_i$ of the entries of the overlap matrix are $i$, exactly $n-v = n-v(\bfr)$ entries are $1$, and the remaining entries are $0$.

Conversely, given such a matrix $\mc{M}$, the number of pairs $(\pi_1, \pi_2)$ with overlap matrix equal to $\mc{M}$ is given by the multinomial coefficient 
\begin{equation*}\frac{n!}{\prod_{i=2}^{a} i!^{r_i}}.
\end{equation*}
This is because, given the matrix $\mc{M}$ and $n$ vertices, we must pick $r_i$ sets of $i$ vertices that correspond to the $i$-entries in $\mc{M}$ for each $i$, as well as $n-v$ single vertices for each of the $1$-entries, and then this exactly defines the two ordered $k$-equipartitions. Given $\bfr$, denote by $M_{\bfr}$ the number of corresponding matrices and observe that
\begin{equation}
\label{relationpm}
 P_{\bfr} = \frac{n!}{\prod_{i=2}^{a} i!^{r_i}}  M_{\bfr}.
\end{equation}
Thus to bound $Q_\bfr= P_\bfr / P^2$, it suffices to bound $M_\bfr$ and we do so in following way. Take an empty $k \times k$-matrix, and write the number $2$ in $r_2$ empty slots, write the number $3$ in $r_3$ empty slots, and so on. There are at most
\begin{align}
{k^2 \choose r_2} {k^2 \choose r_3} \dots  {k^2 \choose r_a} &\le  \frac{k^{2\sum_{i=2}^a r_i}}{\prod_{i=2}^a r_i!}\label{placeentries}
\end{align}
ways to do this. The rest of the matrix has entries $0$ and $1$, and the number of ways to fill in these entries is bounded by the total number of $k \times k$ 0-1 matrices where the row and column sums are given by $\ceilnk$ or $\floornk$ minus the values of the entries that are already written in these rows and columns. Note that we are of course overcounting $M_{\bfr}$, since not all placements of the numbers $2$, $3$, \dots, are valid, and not all 0-1 matrices are possible afterwards, but this will be insignificant.

To estimate the number of 0-1 matrices with prescribed row and column sums, we use the following result of McKay (\cite{mckay1984asymptotics}, see also \cite{greenhill2006asymptotic}).
\begin{theorem}
\label{mckaymatrices}
 Let $N(\mathbf{s}, \mathbf{t})$ be the number of $m \times n$ 0-1 matrices with row sums $\mathbf{s}=(s_1, \dots, s_m)$ and column sums $\mathbf{t}=(t_1, \dots, t_n)$. Let $S= \sum_{x=1}^m s_x$, $s = \max_x s_x$, $t = \max_y t_y$, $S_2 = \sum_{x=1}^m s_x (s_x -1)$ and $T_2 = \sum_{y=1}^n t_y (t_y -1)$.

If $S \rightarrow \infty$ and $1 \le \max\{s,t\}^2 < cS$ for some constant $c < \frac{1}{6}$, then
\[
 N(\mathbf{s}, \mathbf{t}) = \frac{S!}{\prod_{x=1}^m s_x! \, \prod_{y=1}^n t_y!} \exp \left( - \frac{S_2 T_2}{2S^2} + O \left( \frac{\max\{s,t\}^4}{S}\right) \right).
\]
\end{theorem}

Having written the numbers $2, \dots, \ceilnk$ in the matrix, the remaining 0-1 entries must be placed so that the rows sum to $\mathbf{s}=(s_1, \dots, s_k)$, and the columns sum to $\mathbf{t}=(t_1, \dots, t_k)$, where $s_x, t_y \le \ceilnk$ for all $x,y$. The exact values for $s_x$ and $t_y$ depend on the placement of the numbers $2, \dots, \ceilnk$. In the terminology of Theorem~\ref{mckaymatrices}, we have $S= n-v \ge (1-c)n \rightarrow \infty$ and $1 \le \max\{s,t\}^2 \le \ceilnk^2=O(\log^2 n) = o(S)$, so we can apply Theorem \ref{mckaymatrices}:
\begin{align}
N(\mathbf{s}, \mathbf{t}) &=  \frac{(n-v)!}{\prod_{x=1}^k s_x! \, \prod_{y=1}^k t_y!} \exp \left( - \frac{\sum_{x=1}^k s_x (s_x -1)\sum_{y=1}^k t_y (t_y -1)}{2(n-v)^2} + O \parenth{\frac{\log^4 n}{n}}\right).\label{nstfirst1}
\end{align}

As $\sum_{x=1}^k s_x = n-v$, applying Jensen's inequality with the convex function $x(x-1)$ gives
\begin{align}
\sum_{x=1}^k s_x (s_x -1) & \ge k  \left( \frac{n-v}{k} \right) \left(\frac{n-v}{k} -1 \right)= \left( n-v \right) \left(\frac{n-v}{k} -1 \right), \label{intermed}
\end{align}
and the corresponding inequality also holds for $\sum_{y=1}^k t_y (t_y -1)$.

The sequence $s_1, \dots, s_k$ can be obtained from the sequence $\ceilnk$, $\ceilnk$, \dots, $\floornk$ ($k_1$ times $\ceilnk$ and $k_2$ times $\floornk$) by successively subtracting the number $2$ from $r_2$ members of the sequence, the number $3$ from $r_3$ members of the sequence, and so on.

The product $\prod_{x=1}^k s_x!$ can then be obtained from the product $\ceilnk!^{k_1} \floornk!^{k_2}$ by removing the corresponding $v= \sum_{i=2}^a ir_i$ factors of the factorials. If $\sum_{i=2}^a r_i \le k_1$, the product of these factors is maximal if $s_x = \ceilnk -i$ for exactly $r_i$ values $x$ for all $i\ge 2$. For all remaining values $x$, $s_x = \ceilnk$ or $s_x = \floornk$. Therefore, in this case
\begin{align*}
 \prod_{x=2}^k s_x! \ge& \ceilnk!^{k_1} \floornk!^{k_2} \cdot \prod_{i=2}^a \frac{\left( \ceilnk-i\right)!^{r_i}}{\ceilnk!^{r_i}}. 
\end{align*}
Note that the above remains valid if $\sum_{i=2}^a r_i > k_1$ --- it is just not tight in this case. If $\left \lceil \frac{n}{k} \right \rceil <a$, there are no parts of size $i$ in the partition for $\ceilnk <i \le a$, so $r_i=0$ and the above is still well-defined as there are no terms for such $i$. The corresponding inequality of course also holds for $ \prod_{y=2}^k t_y!$. Together with (\ref{nstfirst1}) and (\ref{intermed}), this gives
\begin{align*}
 N(\mathbf{s}, \mathbf{t}) \lesssim&  \frac{(n-v)!}{\ceilnk!^{2k_1} \floornk!^{2k_2}} 
\prod_{i=2}^a \frac{\ceilnk!^{2r_i}}{\left( \ceilnk-i\right)!^{2r_i}} \exp \left( -\frac{1}{2} \left(\frac{n-v}{k} -1 \right)^2 \right). 
\end{align*}
Using (\ref{relationpm}) and (\ref{placeentries}), we have
\begin{align}
 P_\bfr \lesssim  \frac{(n-v)!n!}{\ceilnk!^{2k_1} \floornk!^{2k_2}}\prod_{i=2}^a \frac{k^{2r_i}\ceilnk!^{2r_i}}{i!^{r_i} r_i!\left( \ceilnk-i\right)!^{2r_i}}\exp \left( -\frac{1}{2} \left(\frac{n-v}{k} -1 \right)^2 \right). \label{lala1}
\end{align}
Note that by Stirling's formula $n! \sim \sqrt{2\pi n} n^n/e^n$, and using $1+x \le e^x$,
\begin{align*}
  \frac{(n-v)!}{n!} &\lesssim \frac{(n-v)^{n-v}e^v}{n^{n}} = n^{-v} \left( 1-\frac{v}{n}\right)^{n-v} e^v \le n^{-v} e^{v^2/n} = n^{-v} e^{\rho v}.
\end{align*}
Together with (\ref{defofq}), (\ref{lala1}) and (\ref{Pdefinition}), and as $v= \sum_{i=2}^{a}i r_i$, this gives
\begin{align*}
Q_{\bfr} =\frac{P_{\bfr}}{P^2}  \lesssim&
\prod_{i=2}^a \left( \frac{1}{r_i!}\left(\frac{e^{\rho i}k^{2}\ceilnk!^{2}}{n^i i!\left( \ceilnk-i\right)!^{2}}\right)^{r_i}\right)\exp \left( -\frac{1}{2} \left(\frac{n-v}{k} -1 \right)^2 \right).
\end{align*}
Recalling that $d = \sum_{i=2}^a {i \choose 2} r_i$, and that by (\ref{keyceilnk}) from Section \ref{sectionkeyfacts}, $\ceilnk \le a$,
\begin{align*}
Q_{\bfr} b^d \lesssim 
 \prod_{i=2}^a \left(\frac{1}{r_i!} \left(\frac{e^{\rho i}b^{i \choose 2} k^{2} \ceilnk!^{2}}{ n^i i!\left( \ceilnk-i\right)!^{2}}\right)^{r_i} \right) \exp \left( -\frac{1}{2} \left(\frac{n-v}{k} -1 \right)^2 \right).
\end{align*}
\end{proof}
For the remainder of Section \ref{sectiontypicalrange}, we will only consider the case $\ftn(n)=\epsilon$ for all $n$ and $\bfr \in \mc{R}_1$. We first make some simplifications. Note that
\[
 \exp \left( -\frac{1}{2} \left(\frac{n-v}{k} -1 \right)^2 \right) \le 1.
\]
Furthermore, $\frac{\ceilnk!}{\parenth{\ceilnk-i}!} \le \frac{a!}{\parenth{a-i}!} $ for all $i$, so Lemma \ref{quotablelemma1} gives 
\[
 Q_{\bfr} b^d \le \prod_{i=2}^a \left(\frac{1}{r_i!} \left(\frac{e^{\rho i}b^{i \choose 2} k^{2} a!^{2}}{ n^i i!\left( a-i\right)!^{2}}\right)^{r_i} \right).
\]
Therefore, letting
\[
 T_i :=\frac{e^{\rho i}b^{i \choose 2} k^{2}a!^{2}}{ n^i i!\left( a-i\right)!^{2}},
\]
we have
\begin{align*}
Q_{\bfr} b^d \lesssim  \prod_{i=2}^a \frac{T_i^{r_i}}{r_i!}. 
\end{align*}
By (\ref{keyceilnk}), $\ceilnk \le a$. If $\ceilnk < a$, then there are no parts of size $i$ for $\ceilnk <i \le a$, so $r_i=0$. Therefore,
\begin{align}
Q_{\bfr} b^d \lesssim  \prod_{i=2}^\ceilnk \frac{T_i^{r_i}}{r_i!}. \label{eq6}
\end{align}
Note that the terms $T_i$ still depend on $\bfr$, but only through $\rho(\bfr)$. The next lemma ensures that the terms $T_i$ are small enough as long as $\rho \le c$; the proof is given in the appendix. Let \[
 c_5 = \min \left \lbrace\frac{1}{10}, \frac{c}{2 \log b},\frac{1-c}{2\log b}\right\rbrace \in (0,1),\]
where $c$ is defined in (\ref{defofc}). 
\begin{lemma}
\label{lemmaatmostnconstant}
Suppose $\ftn(n)=\epsilon$ for all $n$. If $\bfr \in \mc{R}_1$ and $n$ is large enough, then for all $3 \le i \le \ceilnk-1$, 
\begin{equation*}
T_i \le n^{-c_5},
\end{equation*}
and for $i\in \left\lbrace 2, \ceilnk \right\rbrace$,
\[T_i \le n^{1-c_5}.\] \qed
\end{lemma}

Let
\[R= \sum_{i=2}^a r_i\]
denote the \emph{total number of overlap blocks}. The following lemma, which is also proved in the appendix, gives a bound for the quantity appearing in (\ref{eq6}) in terms of $R$ instead of the individual $r_i$'s.
\begin{lemma}
 \label{contributionnottoolarge}
If $\ftn(n)=\epsilon$ for all $n$ and if $n$ is large enough, then for all $\bfr \in \mc{R}_1$, 
\begin{equation}
 Q_\bfr b^d \lesssim  n^{-c_5R/2}  \exp \parenth{ \frac{n}{\log^{9} n}},
\label{eq7}
\end{equation}
where $c_5 >0$ is the constant from Lemma \ref{lemmaatmostnconstant}. \qed
\end{lemma}
Now we are finally ready to sum (\ref{eq7}) over all $\bfr \in \mc{R}_1$. For this, note that if $n$ is large enough, then given $R$, there are at most $(2e \log_b n)^{R}$ ways to select $r_2, \dots, r_a$ such that $\sum_{i=2}^a r_i = R$. This is because there are
\begin{align*}
{{R+a-2} \choose R} &\le \left(\frac{e \left(R+a-2\right)}{R} \right)^{R} \le \left(e\left(1+a-2\right) \right)^{R} \le (2e \log_b n)^{R}
\end{align*}
ways to write $R$ as an ordered sum with $a-1$ nonnegative summands.

Using this and Lemma \ref{contributionnottoolarge}, if $n$ is large enough and $\ftn(n)=\epsilon$ for all $n$, we can now simply take the sum over $R$.
\begin{align*}
 \sum_{\bfr \in \mc{R}_1} Q_\bfr b^d &\lesssim \sum_{R=0}^{\infty} \parenth{\parenth{2e \log_b n}^{R}  n^{-c_5R /2} \exp \parenth{ \frac{n}{\log^{9}n}}} \\
&=  \exp \parenth{ \frac{n}{\log^{9}n}} \sum_{R=0}^{\infty} \parenth{\frac{2e \log_b n}{n^{c_5/2}}}^{R}  \le  2\exp \parenth{ \frac{n}{\log^9 n}}.
\end{align*}
Therefore, we have that $\sum_{\bfr \in \mc{R}_1} Q_\bfr b^d \le \exp \parenth{ \frac{n}{\log^8 n}}$ for $n$ large enough if $\ftn(n)=\epsilon$ for all $n$, as required.

\subsection{Pairs of partitions with many small overlap blocks}
\label{middle}
In this section, we will bound the contribution to the sum (\ref{thesum}) from those overlap sequences $\bfr$ with $\rho=\rho(n):=v/n \ge c$, but where there are either still many singletons which are not involved in the overlap (so $n-v$ is large) or many vertices in `small' overlap blocks of size at most $0.6\gamma$. More specifically, fix a constant $0<{c'}<1$ and consider only those $\bfr$ with $\rho > c$ such that there are at least ${c'} n$ singletons or at least ${c'} n$ vertices in overlap blocks of size at most $0.6 \gamma$:
\[
 \mc{R}_2^{c'} = \left \lbrace \bfr \mid \rho >c  \wedge \parenth{\sum_{2 \le i \le 0.6\gamma} i r_i \ge {c'} n  \vee \rho \le 1-{c'}}\right \rbrace.
\]
We will prove that for any fixed ${c'} \in (0,1)$, the contribution to the sum (\ref{thesum}) from these overlap sequences is negligible: if $\ftn(n)=\epsilon$ for all $n$ in the definition of $k$, then
\[\sum_{\bfr \in \mc{R}_2^{c'}} Q_\bfr b^d =o(1).\]
To do this, we will generate all pairs of partitions in this range by taking the first partition, grouping the vertices into subsets of its parts which will form the singletons and overlap blocks, and rearranging them into $k$ sets to get the new partition. If we do this according to some $\bfr \in \mc{R}_2^{c'}$, then we can bound the number of ways to generate another partition as well as the number of overlapping edges $d$ between the two partitions.

\subsubsection{Preliminaries}
We first need some notation and preliminary results. Since some of our bounds need to be extremely accurate, we distinguish between parts of size $a$ and parts of size at most $a-1$. By (\ref{keyceilnk}), $\ceilnk \le a$, so there may of course be no parts of size $a$ at all. Fix an \emph{arbitrary ordered $k$-equipartition} $\pi_1$, and let 
\[
\mc{P}_2^{c'} = \left\lbrace \text{ordered $k$-equipartitions $\pi_2$ such that }\bfr (\pi_1,\pi_2) \in \mc{R}_2^{c'} \right \rbrace.
\]
Given $\pi_2 \in \mc{P}_2^{c'}$, let 
\begin{align*}
 V_1 =& \text{ set of vertices in the overlap of $\pi_1$ and $\pi_2$ that are in parts of size $a$ in $\pi_1$}\\ 
 V_2 =& \text{ set of vertices in the overlap of $\pi_1$ and $\pi_2$ that are in parts of size at most $a-1$}\\
&\text{ in $\pi_1$}\\
 D_1 =& \text{ set of overlapping forbidden edges between vertices in $V_1$}\\
 D_2 =& \text{ set of overlapping forbidden edges between vertices in $V_2$.}
\end{align*}
For $i \in \lbrace 1,2 \rbrace$, let $v_i=|V_i|$ and $d_i=|D_i|$, so $v_1+v_2=v$ and $d_1+d_2=d$.

Given $\pi_1$ and $\pi_2$, we define the \emph{overlap graph}  of $\pi_1$ and $\pi_2$ as the union of all the vertices in overlap blocks together with all the common forbidden edges. By definition, the overlap graph is a disjoint union of cliques, each containing between $2$ and $\ceilnk$ vertices. Note that the vertex set is exactly $V_1 \cup V_2$ and the edge set exactly $D_1 \cup D_2$. Denote by $\mathbf{g}=(g_j)_{j=1}^{v_1}$ the degree sequence in the overlap graph of the vertices in $V_1$. Then $g_j \le a-1$ for all $j$, so
\[
 2d_1 = \sum_{j=1}^{v_1} g_j \le  v_1 \left(a-1\right).
\]
Similarly,
\[
 2d_2 \le v_2 \left({a-2}\right).
\]
To quantify how close the overlap graph of $\pi_1$ and $\pi_2$ is to consisting only of cliques (or overlap blocks) of sizes $a$ or $a-1$, we define the following simple edge density parameters: let
\begin{align*}
\beta_1&= \frac{2d_1}{v_1   \left( a-1\right)} \le 1\nonumber \\
\beta_2&= \frac{2d_2}{v_2   \left( a-2\right)} \le 1.
\label{defofbetai}
\end{align*}
If $\beta_1$ and $\beta_2$ are close to $1$, then the overlap of $\pi_1$ and $\pi_2$ consists almost entirely of large overlap blocks which are almost entire parts. We will now give some simple bounds for the number of vertices in smaller overlap blocks in terms of $\beta_1$ and $\beta_2$.

Indeed, for $x \in (0,1)$, we denote by $w_{x,1}$  the proportion of vertices in $V_1$ which have degree at most $x  \left( a-1\right)$ within the overlap graph, i.e.,
\[
 w_{x,1} = \frac{\text{\# $j$ with }g_j \le x  \left( a-1\right)}{v_1}.
\]
Then, as $g_j \le a-1$ for all $j$, for any $x \in (0,1)$,
\[
 \beta_1 v_1 \left( a-1\right) = 2d_1= \sum_{j=1}^{v_1}g_j \le w_{x,1} v_1 x  \left( a-1\right)+(1-w_{x,1})v_1   \left( a-1\right),
\]
so
\begin{equation}
\label{wxboundi}
 w_{x,1} \le \frac{1-\beta_1}{1-x}.
\end{equation}
Similarly, define $w_{x,2}$ as the proportion of vertices in $V_2$ that have degree in the overlap graph of at most $x  \left( {a-2}\right)$. Analogously, we have
\begin{align*}
 w_{x,2} \le \frac{1-\beta_2}{1-x}.
\end{align*}

Next, we need a bound for the total number of overlap blocks. As in the previous section, let 
\begin{equation}R = \sum_{i=2}^{a}r_i \label{defofR}\end{equation}
 denote the number of overlap blocks. Let $R_1$ and $R_2$ denote the number of overlap blocks in parts of size $a$ and of size at most ${a-1}$ in $\pi_1$, respectively, so $R=R_1+R_2$. 

Note that
\[
R_1= \sum_{j=1}^{v_1} \frac{1}{g_j+1},
\]
as every overlap block of $s$ vertices contributes exactly $s$ instances of the summand $\frac{1}{s}$. For any $0<x<y<1$, there are $ w_{x,1}v_1 $ values $i$ such that $1 \le g_i \le x \left( a-1\right)$, at most $w_{y,1}v_1$ values $i$ such that $x \left( a-1\right) <g_i \le y \left( a-1\right)$, and for the remaining values $i$, $g_i > y \left(a-1\right)$. Therefore,
\begin{align*}
 R_1&\le \frac{w_{x,1}v_1}{2}+\frac{w_{y,1}v_1}{x \left( a-1\right)+1}+ \frac{v_1}{y \left( a-1\right)+1} \le \frac{w_{x,1}v_1}{2}+\frac{w_{y,1}v_1}{x \gamma}+ \frac{v_1}{y\gamma}.
\end{align*}
Using (\ref{wxboundi}), it follows that
\begin{align}
R_1&\le \frac{(1-\beta_1)v_1}{2(1-x)} + \frac{(1-\beta_1)v_1}{x(1-y)\gamma-1}+\frac{v_1}{y\gamma-1} \label{eqforr1},
\end{align}
where in the last term $y\gamma$ was replaced by $y\gamma-1$ so that the corresponding expression holds for $R_2$ as well. Indeed, we can see that
\begin{align}
 R_2& \le \frac{(1-\beta_2)v_2}{2(1-x)} + \frac{(1-\beta_2)v_2}{x(1-y)\gamma-1}+\frac{v_2}{y\gamma-1}  \label{eqforr2}.
\end{align}

In the following lemma, which is proved in the appendix, we give some weaker but more convenient conditions for $\pi_2$ and show that any $\pi_2 \in \mc{P}_2^{c'}$ meets one of these conditions.
\begin{lemma}
\label{implies}
If $\pi_2 \in \mc{P}_2^{c'}$ and $n$ is large enough, then at least one of the following three conditions applies.
\begin{enumerate}[I)]
 \item $v_1 \ge \frac{n}{\parenth{\log \log n}^2}$ and $\beta_1 \le 1-\frac{(\log \log n)^4}{\log n}$. \label{condition1}
 \item $v_2 \ge \frac{n}{\parenth{\log \log n}^2}$ and $\beta_2 \le 1-\frac{(\log \log n)^4}{\log n}$. \label{condition2}
 \item Neither \ref{condition1} nor \ref{condition2} holds, and $c <\rho \le 1-{c'}$.\qed
\end{enumerate} 
\end{lemma}
Still fixing the arbitrary ordered $k$-equipartition $\pi_1$, let
\begin{align*}
 \mc{P}^{\text{I}} &= \left \lbrace\text{ordered $k$-equipartitions $\pi_2$ such that $v_1 \ge \frac{n}{\parenth{\log \log n}^2}$  and $\beta_1 \le 1-\frac{(\log \log n)^4}{\log n}$}\right \rbrace\\
 \mc{P}^{\text{II}} &= \left \lbrace \text{ordered $k$-equipartitions $\pi_2$ such that $v_2 \ge \frac{n}{\parenth{\log \log n}^2}$ and $\beta_2 \le 1-\frac{(\log \log n)^4}{\log n}$}\right \rbrace \\
 \mc{P}^{\text{III}} &= \left \lbrace  \text{ordered $k$-equipartitions $\pi_2$ such that $c < \rho \le 1-{c'}$}\right \rbrace \setminus  \mc{P}^{\text{I}} \setminus  \mc{P}^{\text{II}},
\end{align*}
where $v_i$, $\beta_i$ and $\rho$ refer to the overlap of $\pi_1$ and $\pi_2$. Then by Lemma \ref{implies} for $n$ large enough,
\[
 \mc{P}_2^{c'} \subset  \mc{P}^{\text{I}} \cup  \mc{P}^{\text{II}}\cup  \mc{P}^{\text{III}}.
\]
For an overlap sequence $\bfr$, denote by $P'_{\mathbf{r}}$ the number of ordered $k$-equipartitions with overlap $\mathbf{r}$ with $\pi_1$. Then by the definition (\ref{defofq}) of $Q_\bfr$,
\begin{align}
Q_{\mathbf{r}} &=  \frac{P_{\mathbf{r}}}{P^2}= \frac{P'_{\mathbf{r}}}{P}  .\label{QandP}
\end{align}
Using (\ref{pasymp}) in the last step, if $n$ is large enough,
\begin{align}
\sum_{\bfr \in \mc{R}_2^{c'}} Q_{\bfr} b^d &  = \sum_{\bfr \in \mc{R}_2^{c'}} \frac{P'_\bfr}{P} b^d= \sum_{\pi_2 \in \mc{P}_2^{c'}} P^{-1} b^{d(\pi_1,\pi_2)} \le \sum_{\pi_2 \in \mc{P}^{\text{I}} \cup  \mc{P}^{\text{II}}\cup  \mc{P}^{\text{III}}} P^{-1} b^{d(\pi_1,\pi_2)} \nonumber \\
&= \sum_{\pi_2 \in \mc{P}^{\text{I}} \cup  \mc{P}^{\text{II}}\cup  \mc{P}^{\text{III}}} k^{-n} b^{d(\pi_1,\pi_2)} \exp(o(n)), \label{conversion}
\end{align}
where $d(\pi_1,\pi_2):=d(\bfr)$ if $\bfr$ is the overlap sequence of $\pi_1$ and $\pi_2$.

We will now generate and count all $\pi_2 \in \mc{P}^{\text{I}} \cup  \mc{P}^{\text{II}}\cup  \mc{P}^{\text{III}}$. Starting with $\pi_1$, we first subdivide the parts into overlap blocks and singletons. Then we arrange those overlap blocks and singletons into $k$ new parts to generate $\pi_2$, and sum the resulting $b^{d(\pi_1,\pi_2)}$.

\subsubsection{Contribution from Cases I and II}
\label{contribution12}
 We start by generating the partitions in $\mc{P}^{\text{I}} \cup \mc{P}^{\text{II}}$ according to the following strategy. We group the vertices into subsets of the parts of $\pi_1$ which form the overlap blocks and singletons for the overlap with $\pi_2$, and give a bound for the number of ways this can be done in Lemma \ref{choppinguplemma1}. Then we sort the overlap blocks and singletons into the $k$ parts of $\pi_2$. If there are $R$ overlap blocks and $n-v$ singletons, then there are at most $k^{n-v+R}$ choices for this. Considering (\ref{conversion}), the term $k^n$ cancels out with $k^{-n}$, leaving just $k^{-v+R}b^d$ multiplied by the bound from Lemma \ref{choppinguplemma1} as an upper bound for (\ref{conversion}). If Cases I or II apply and we also use the bounds (\ref{eqforr1}) and (\ref{eqforr2}) for $R$, then $d_i$ will be small enough in comparison to $v_i$ for at least one $i \in \lbrace 1,2 \rbrace$ so that  $k^{-v+R}$ is much smaller than $b^d$, allowing us to bound the total contribution from Cases I and II to (\ref{conversion}) and thereby to (\ref{thesum}) by $o(1)$.

\begin{lemma}
 \label{choppinguplemma1}
Denote by $S$ the number of ways $n$ vertices can be partitioned into subsets (of any size and number) of the parts of $\pi_1$. Then
\[
 S\le \exp\left(O \left( n \log \log n\right)\right).
\]
\end{lemma}
\begin{proof}
If we sort the $n$ vertices into $a$ containers, this defines a subdivision of $\pi_1$ by letting all vertices be in the same set that are in the same part of $\pi_1$ and in the same container. Conversely, any possible subdivision of $\pi_1$ can be obtained in this way, since every part can only be partitioned into at most $a$ non-empty sets. Therefore, as $a = O(\log n)$,
\[
 S \le a^n =\exp\left(O \left( n \log \log n\right)\right).\]
\end{proof}
We are now ready to show that the contribution to (\ref{conversion}) from all ordered $k$-equipartitions in $\mc{P}^{\text{I}} \cup \mc{P}^{\text{II}}$ to (\ref{conversion}) is $o(1)$.
\begin{lemma}
\label{lemmar2}
 \[\sum_{\pi_2 \in \mc{P}^{\text{I}} \cup  \mc{P}^{\text{II}}} k^{-n} b^{d(\pi_1,\pi_2)} \exp(o( n)) =o(1).
\]
\end{lemma}
\begin{proof}
Fix $v_1$, $v_2$, $d_1$ and $d_2$ so that I or II holds. 
Let
\[
 \mc{P}(v_1,v_2,d_1,d_2) = \left \lbrace \pi_2 \in  \mc{P}^{\text{I}} \cup  \mc{P}^{\text{II}} \mid v_i(\pi_1,\pi_2)=v_i, d_i(\pi_1,\pi_2)=d_i, i=1,2\right \rbrace.
\]
Arrange the $n$ vertices into singletons and overlap blocks that are subsets of the parts of $\pi_1$ in accordance with $v_1$, $v_2$, $d_1$ and $d_2$. Now that we know the $R$ overlap blocks and $n-v$ singletons, the number of ordered $k$-equipartitions $\pi_2$ with these overlap blocks with $\pi_1$ is at most $k^{n-v+R}$, since we need to sort $n-v$ singletons and $R$ overlap blocks into $k$ parts.

Therefore, letting $x=\frac{1}{4}$ and $y=1-\frac{1}{\log \log n}$, then with (\ref{eqforr1}), (\ref{eqforr2}) and Lemma \ref{choppinguplemma1},
\begin{align}
\sum_{\pi_2 \in \mc{P}(r_1,r_2,d_1,d_2)} &k^{-n} b^{d(\pi_1,\pi_2)}\le Sk^{-n+n-v+\sum_{i=1}^2 \parenth{\frac{2(1-\beta_i)v_i}{3} + \frac{4(1-\beta_i)v_i}{(1-y)\gamma-4}+\frac{v_i}{y\gamma-1}} } b^{d_1+d_2} \exp \parenth{o(n)} \nonumber\\
&\le k^{\sum_{i=1}^2 \parenth{-v_i+\frac{2(1-\beta_i)v_i}{3} + \frac{4(1-\beta_i)v_i}{(1-y)\gamma-4}+\frac{v_i}{y\gamma-1}} }   b^{d_1+d_2} \exp\left(O (n \log\log n) \right).\label{calc1}
\end{align}
Note that as by (\ref{key2}) from Section \ref{sectionkeyfacts}, $b^{\frac{\gamma}{2}}\sim k$, and since $\beta_i v_i \le n$ for $i\in \lbrace 1,2 \rbrace$ and $a= \left \lfloor \gamma \right \rfloor+1\le \gamma+1$,
\begin{equation*}
b^{d_1+d_2}= b^{\beta_1 v_1 \frac{a-1}{2}+\beta_2 v_2 \frac{a-2}{2}} \le   b^{(\beta_1 v_1+\beta_2 v_2) \frac{\gamma}{2}}\le k^{\beta_1 v_1+\beta_2 v_2}  \exp(o(n)).
\end{equation*}
By (\ref{key7}) and since $v_i \le n$, $k^{\frac{v_i}{y\gamma-1}} \le \exp (O(n))$ for $i\in\lbrace1,2\rbrace$, and therefore (\ref{calc1}) becomes
\begin{align*}
\sum_{\pi_2 \in \mc{P}(r_1,r_2,d_1,d_2)} k^{-n} b^{d(\pi_1,\pi_2)}&\le k^{-\sum_{i=1}^2 \parenth{v_i(1-\beta_i)\left(\frac{1}{3} -\frac{4}{(1-y)\gamma-4}\right)}}  \exp\left(O (n \log\log n)\right) .
\end{align*}
Recall that $y=1-\frac{1}{\log \log n}$, so $(1-y)\gamma \rightarrow \infty$, and we have $\frac{1}{3} -\frac{4}{(1-y)\gamma-4} \ge \frac{1}{4}$ for $n$ large enough. Since I or II holds, there is an $i \in \lbrace 1,2 \rbrace$ such that $(1-\beta_i) v_i \ge \frac{n \parenth{\log \log n}^2}{\log n}$, so by (\ref{key7}),
\begin{align*}
\sum_{\pi_2 \in \mc{P}(r_1,r_2,d_1,d_2)} k^{-n} b^{d(\pi_1,\pi_2)} &\le k^{-\frac{n (\log \log n)^2}{4\log n} }\exp \parenth{ O (n \log\log n)}.
\end{align*}
As $f = O(n\log n)$ by (\ref{key6}), and since $v_i \le n$ and $d_i \le f$ for $i\in \lbrace 1,2 \rbrace$, there are only $O(n^4 \log^2 n)$ choices for the values of $v_i \le n$ and $d_i$ for $i\in \lbrace 1,2 \rbrace $. Hence,
\begin{align*}
\sum_{\pi_2 \in \mc{P}^{\text{I}} \cup  \mc{P}^{\text{II}}} k^{-n} b^{d(\pi_1,\pi_2)} \exp(o(n)) &\le k^{-\frac{n(\log \log n)^2}{4\log n}}\exp\parenth{O(n\log \log n)}\\
&= \exp\parenth{-\Theta\parenth{n(\log \log n)^2}}=o(1).
\end{align*}
\end{proof}
\subsubsection{Contribution from Case III}
\label{contribution3}
We have to be a bit more careful in the case where neither I nor II holds. We will proceed similarly as in the proof of Lemma \ref{lemmar2}: we subdivide $\pi_1$ into subsets and then sort the singletons and overlap blocks into the $k$ parts to form the new partition $\pi_2$. Since for both $i \in \lbrace 1,2 \rbrace$, $\beta_i$ is either close to $1$ or $v_i$ is negligibly small, most of the overlap blocks will be almost entire parts of $\pi_1$. If we place those large overlap blocks first, they occupy a constant fraction of about $\rho k$ of the $k$ parts almost entirely, so the remaining roughly $(1-\rho)n$ vertices and smaller overlap blocks have fewer choices left, namely only about $(1-\rho)k$ choices each. This will give an additional factor of about $(1-\rho)^{(1-\rho)n}$. Almost everything else will turn out to be subexponential, except for a term which is about $b^{(1-\Delta) v/2}$. As $v=\rho n$, this will result in a total bound which is roughly of the form $b^{-((1-\rho)\log_b(1-\rho)-(1-\Delta)\rho/2)n}$. Comparing the exponent of this expression with condition (\ref{aste}) from Theorem \ref{maintheorem} (using the technical lemmas we proved in Section \ref{technicallemmas}), we will show that the sum is $o(1)$ for $c<\rho < 1-{c'}$.

Instead of Lemma \ref{choppinguplemma1}, which gave a fairly slack bound on the number of ways the vertices may be arranged into subsets of the parts of $\pi_1$, we now need a more accurate bound. The following lemma ensures that if Condition III applies, the number of ways to subdivide $\pi_1$ is subexponential.
\begin{lemma}
 \label{choppinguplemma3}
Fix integers $v_1$, $v_2$, $d_1$, $d_2$ so that I and II do not hold as above. Denote by $S(v_1,v_2,d_1,d_2)$ the number of ways the vertices can be partitioned into subsets of the parts of $\pi_1$ which form overlap blocks and singletons according to $v_i$ and $d_i$, $i \in \lbrace 1,2 \rbrace$. Then there is a function $S'= S'(n)$ which does not depend on $v_i$ or $d_i$, $i=1,2$, such that
\[
 S(v_1,v_2,d_1,d_2) \le S'\le\exp \left( o(n)\right).
\]
\end{lemma}
\begin{proof}
We first split up the parts of size $a$ (if any such parts exist). Since Condition I does not hold, either $v_1 < \frac{n}{\parenth{\log \log n}^2}$ or $\beta_1 > 1-\frac{(\log \log n)^4}{\log n}$.

In the first case, select the $v_1 < \frac{n}{\parenth{\log \log n}^2}=o(n)$ vertices which form the overlap blocks in parts of size $a$. Using (\ref{keychoose}) from Section \ref{sectionkeyfacts}, there are at most
\[
 {n \choose v_1} \le {n \choose \left\lfloor \frac{n}{(\log\log n)^2} \right\rfloor} \le \exp(o(n))
\]
ways to do this. All the other vertices in parts of size $a$ must be singletons. To find out how the $v_1$ vertices are arranged into overlap blocks, we can proceed as in the proof of Lemma \ref{choppinguplemma1}: sort the $v_1$ vertices into $a$ containers, and let those vertices be in the same overlap block that are in the same container and in the same part of $\pi_1$. There are
\[
 a^{v_1} \le a^\frac{n}{(\log\log n)^2}= \exp\parenth{O \parenth{\frac{n}{\log\log n}}} \le  \exp (o(n))
\]
possibilities for this, so altogether there are $\exp (o(n))$ ways to split up the parts of size $a$ in the case $v_1 < \frac{n}{\parenth{\log \log n}^2}$.

In the second case, we have $\beta_1 > 1-\frac{(\log \log n)^4}{\log n}$. Let $x=1-\frac{(\log\log n)^2}{(\log n)^{1/2}}$, then by (\ref{wxboundi}), if $\pi_2$ overlaps with $\pi_1$ according to $v_i$ and $d_i$, then
\[
 w_{x,1} \le \frac{\frac{(\log \log n)^4}{\log n}}{\frac{(\log\log n)^2}{(\log n)^{1/2}}}= \frac{(\log\log n)^2}{(\log n)^{1/2}} =: \hat w_x\rightarrow 0.
\]
This means that almost all of the $v_1$ vertices in the overlap must be arranged into large overlap blocks of size greater than $x \left( a-1\right)+1$. As $x \rightarrow 1$, we can assume $x>2/3$. Therefore, any part of $\pi_1$ contains at most one such large overlap block, and we can group the vertices in parts of size $a$ into overlap blocks and singletons in the following way.
\begin{itemize}
 \item First we select the parts which contain large overlap blocks. There are at most
\[
 2^k = \exp \parenth{ O \left( n / \log n\right)} = \exp (o(n))
\]
choices.
\item Next, given these $k' \le k$ parts, we pick the vertices within the parts that are not in the large overlap blocks of size greater than $x \left( a-1\right)+1$. Since $x \rightarrow 1$, there are at most 
\[k'(a-x(a-1)-1)\le(1-x) a k'=o\parenth{a k'} = o(n)
\]
 such vertices. Therefore, there are at most
\begin{align*}
\sum_{l \le (1-x) a k'} {a k' \choose l} &\le \parenth{(1-x) a k'+1}\cdot{a k' \choose {\left\lfloor(1-x) a k'\right\rfloor}} \le n \cdot {n \choose (1-x)n} \le \exp(o(n))
\end{align*}
possibilities for this.
\item Now we know all the large overlap blocks in $V_1$. From the remaining vertices, we choose those vertices that are not singletons, i.e., which are in overlap blocks of size at least $2$, but not in big overlap blocks. There cannot be more than $\hat w_x v_1 \le \hat w_x n=o(n)$ such vertices. Therefore, there are at most
\begin{align*}
  \sum_{j \le\hat w_x n}{n \choose j} &\le  (\hat w_x n+1){n \choose {\left\lfloor w_x n\right \rfloor}} \le  \exp (o(n))
\end{align*}
choices.
\item We have determined all of the large overlap blocks and which of the remaining vertices are singletons and which are in overlap blocks. It only remains to group the vertices that are in overlap blocks into subsets of the parts of $\pi_1$. As in the proof of Lemma \ref{choppinguplemma1}, each such partition into subsets can be obtained by sorting the vertices into $a$ containers, and since there are at most $\hat w_x v_1 \le \hat w_x n$ vertices left, this can be done in at most
\[
 a^{\hat w_x n} = \exp \left(O \left( n \hat w_x \log \log n\right) \right) = \exp (o(n))
\]
 ways.
\end{itemize}
Multiplying everything, and noting that none of the bounds depend on the specific choice of $v_i$ and $d_i$, gives the bound $\exp(o(n))$ for the number of ways we can subdivide the parts of size $a$ in the second case, and hence in both cases.

The bound $\exp (o(n))$ for subdividing the parts of size at most $a-1$ can be proved analogously. Multiplying those two bounds gives $S'=S'(n)$ such that
\[
  S(v_1,v_2,d_1,d_2)\le S'\le\exp (o(n)).
\]
\end{proof}

\begin{lemma}
\label{extralemma1}
Fix $v_1$, $v_2$, $d_1$ and $d_2$ in such a way that I and II do not hold but III does. Let $v=v_1+v_2$ as before, and let
\[
 \mc{P}'(v_1,v_2,d_1,d_2) = \left \lbrace \pi_2 \in  \mc{P}^{\text{III}} \mid v_i(\pi_1,\pi_2)=v_i, d_i(\pi_1,\pi_2)=d_i, i=1,2\right \rbrace.
\]
Then 
\begin{align}
\sum_{\pi_2 \in \mc{P}'(r_1,r_2,d_1,d_2)} k^{-n} b^{d(\pi_1,\pi_2)}\le b^{n(1-\rho)\log_b (1-\rho) +\frac{v_1}{2} - \frac{\Delta v}{2}}\exp(o(n)),
\label{yahoo}
\end{align}
where the function which is implicit in the $o(n)$ term does not depend on our choice of $v_1$, $v_2$, $d_1$ or $d_2$.
\end{lemma}

\begin{proof}
Let $u =1-\frac{(\log \log n)^5}{\log n} \rightarrow 1$. Recall that $\rho=v/n=(v_1+v_2)/n$. 
\begin{claim*}
For any $\pi_2 \in \mc{P}'(v_1, v_2, d_1, d_2)$, there are $(1+o(1))\rho k$ `large' overlap blocks of size at least $u(a-2)$ in the overlap of $\pi_1$ and $\pi_2$.
\end{claim*}
\begin{proof}
Of course there are asymptotically at most $\frac{v}{u(a-2)} \sim \rho k$ such blocks, so we only need to show that there are asymptotically at least $\rho k$ of them.

Note that if $v_i \ge \frac{n}{\parenth{\log \log n}^2}$ for $i\in \lbrace 1,2 \rbrace$, then as I and II do not hold, $\beta_i >1-\frac{(\log\log n)^4}{\log n}$, and therefore,
\begin{equation}\frac{1-\beta_i}{1-u} \le \frac{1}{\log \log n}\rightarrow 0.\label{boundbetaandu}\end{equation}
If $\pi_2 \in \mc{P}'(v_1, v_2, d_1, d_2)$, then by (\ref{wxboundi}), there are at least
\[
 (1-w_{u,1})v_1+(1-w_{u,2})v_2 \ge \sum_{i=1}^2\left(1-\frac{1-\beta_i}{1-u}\right){v_i}
\]
vertices in large overlap blocks of size at least $u(a-2)$. Since no overlap block contains more than $a$ vertices, there are at least
\begin{equation}
 \sum_{i=1}^2\left(1-\frac{1-\beta_i}{1-u}\right)\frac{v_i}{a} \label{zwischen}
\end{equation}
such large overlap blocks. As III holds, $v_1+v_2 = v \ge cn$, so there can be at most one $i\in \lbrace 1,2 \rbrace$ with $v_i < \frac{n}{(\log \log n)^2}$. If this is the case and $j$ is the other element of $\lbrace 1,2\rbrace$, then $v_i \ll v_j \sim v \sim \rho n$, so together with (\ref{boundbetaandu}), (\ref{zwischen}) is 
\[
 o\parenth{\frac{n}{a}}+ \left(1-\frac{1-\beta_j}{1-u}\right)\frac{v_j}{a} =o(k)+(1+o(1))\rho\frac{n}{a} \sim \rho k,
\]
as $\frac{n}{a}\sim k$ by (\ref{key1}). Otherwise, if for both $i \in \lbrace 1,2\rbrace$, $v_i \ge \frac{n}{(\log \log n)^2}$, (\ref{zwischen}) and (\ref{boundbetaandu}) give
\[
 \sum_{i=1}^2\left(1-\frac{1-\beta_i}{1-u}\right)\frac{v_i}{a} \ge \left(1-\frac{1}{\log \log n} \right)\frac{v_1+v_2}{a} \sim \frac{v_1+v_2}{a} \sim \rho k.
\]
So in both cases, there are asymptotically at least $\rho k$ large overlap blocks of size at least $u\left(a-2\right)$. 
\end{proof}

We first subdivide the partition $\pi_1$ into overlap blocks and singletons according to $v_1$, $v_2$, $d_1$, $d_2$ (for which there are $\exp(o(n))$ choices by Lemma \ref{choppinguplemma3}), and then we generate all $\pi_2 \in \mc{P}'(v_1,v_2,d_1,d_2)$. Recall that $R$ was defined in (\ref{defofR}) as the total number of overlap blocks.
\begin{claim*}
 There are at most
\[
 (1-\rho)^{(1-\rho)n}k^{n-v+R} \exp(o(n)) 
\]
other ordered $k$-equipartitions with the given overlap blocks with the original partition $\pi_1$.
\end{claim*}
\begin{proof}
We sort the overlap blocks and singletons into $k$ parts to create a new ordered $k$-equipartition $\pi_2$, and start with the large sets of size at least $u\left(a-2\right)$. By the previous claim, there are $(1+o(1))\rho k$ of them, and each has at most $k$ choices. As $u\rightarrow 1$, we can assume $u>0.6$, so no two large overlap blocks can be assigned to the same part.

After we are finished with the large overlap blocks, the remaining vertices can either be sorted into the small remainder of the $(1+o(1))\rho k$ parts of $\pi_2$ which have been assigned a large block, or they can be sorted into the remaining $(1-\rho+o(1))k$ parts of $\pi_2$.

As $u \rightarrow 1$, we can fit at most $(1+o(1))\rho k  \left(a-u\left(a-2\right)\right)=o(n)$ vertices into the remainder of the parts of $\pi_2$ with large overlap blocks. Therefore, by (\ref{keychoose}) there are at most 
\[
{n \choose o(n)}\le \exp(o(n))
\]
 ways of picking these vertices, and for each there are at most $k$ choices for which part of $\pi_2$ it is assigned to.

There are now at least $n-v-o(n)=(1-\rho+o(1))n$ singletons and overlap blocks left to be assigned to the remaining $(1-\rho+o(1))k$ parts. For each of these there are at most $(1-\rho+o(1))k$ choices.

We have now sorted $R$ overlap blocks and $n-v$ singletons into the $k$ parts, and bounded the number of choices for each by at most $k$, and for $(1-\rho+o(1))n$ of them by $(1-\rho+o(1))k$. Therefore, in total there are at most
\[
 \parenth{1-\rho+o(1)}^{(1-\rho+o(1))n} k^{n-v+R}\le (1-\rho)^{(1-\rho)n}k^{n-v+R} \exp(o(n)) 
\]
ways to build a new partition $\pi_2$ from the given overlap blocks and singletons.
\end{proof}

Now as before, let $x= \frac{1}{4}$ and $y=1-\frac{1}{\log \log n}$. Then as in (\ref{calc1}), by Lemma \ref{choppinguplemma3}, (\ref{eqforr1}) and (\ref{eqforr2}), and since $R=R_1+R_2$,
\begin{align*}
&\sum_{\pi_2 \in \mc{P}'(v_1,v_2,d_1,d_2)} k^{-n} b^{d(\pi_1,\pi_2)}\\
&\le S(v_1,v_2,d_1,d_2)(1-\rho)^{(1-\rho)n}k^{-n+n-v+\sum_{i=1}^2 \parenth{\frac{2(1-\beta_i)v_i}{3} + \frac{4(1-\beta_i)v_i}{(1-y)\gamma-4}+\frac{v_i}{y\gamma-1}}} b^{d_1+d_2} \exp(o(n))\nonumber\\
& \le  (1-\rho)^{(1-\rho)n} k^{\sum_{i=1}^2 \parenth{-v_i+\frac{2(1-\beta_i)v_i}{3} + \frac{4(1-\beta_i)v_i}{(1-y)\gamma-4}+\frac{v_i}{y\gamma-1}} } b^{d_1+d_2}  \exp\left(o(n)\right).\nonumber 
\end{align*}
Note that as by (\ref{key2}), $b^{\frac{\gamma}{2}}\sim k$, and as $a=\left \lfloor \gamma \right \rfloor+1=\gamma-\Delta+1$,
\begin{align*}
b^{d_1+d_2}&= b^{\beta_1 v_1 \frac{a-1}{2}+\beta_2 v_2 \frac{a-2}{2}} =   b^{(\beta_1 v_1+\beta_2 v_2) \frac{\gamma}{2}-\frac{1}{2} \parenth{\Delta \beta_1 v_1+(1+\Delta)\beta_2 v_2}} 
\\
&\le k^{\beta_1 v_1+\beta_2 v_2} b^{-\frac{1}{2} \parenth{\Delta \beta_1 v_1+(1+\Delta)\beta_2 v_2}} \exp(o(n)).
\end{align*}
Since I and II do not hold, $v_i(1-\beta_i)=o(n)$ for $i=1,2$, and therefore,
\[
 b^{d_1+d_2} \le k^{\beta_1 v_1+\beta_2 v_2} b^{-\frac{1}{2} \parenth{\Delta v_1+(1+\Delta)v_2}} \exp(o(n)).
\]
Hence,
\begin{align*}
\sum_{\pi_2 \in \mc{P}'(v_1,v_2,d_1,d_2)} k^{-n} b^{d(\pi_1,\pi_2)}\le  &(1-\rho)^{(1-\rho)n} k^{-\sum_{i=1}^2 \parenth{v_i(1-\beta_i)\left(\frac{1}{3} -\frac{4}{(1-y)\gamma-4}\right)}} k^{\frac{v_1+v_2}{y\gamma-1}} \\
& \cdot b^{-\frac{1}{2} \parenth{\Delta v_1+(1+\Delta)v_2}} \exp\left(o (n)\right) \\
\le & (1-\rho)^{(1-\rho)n} k^{\frac{v_1+v_2}{y\gamma-1}}b^{-\frac{1}{2} \parenth{\Delta v_1+(1+\Delta)v_2}} \exp\left(o (n)\right)
\end{align*}
as $\frac{1}{3} -\frac{4}{(1-y)\gamma-4} >0$ because $(1-y)\gamma \rightarrow \infty$.
Since $\gamma \sim 2 \log_b n$ and $y \rightarrow 1$,
\begin{align*}
 k^{\frac{v_1+v_2}{y\gamma-1}}b^{-\frac{1}{2} \parenth{\Delta v_1+(1+\Delta)v_2}} &\le n^{\frac{v_1+v_2}{y\gamma-1}}b^{-\frac{1}{2} \parenth{\Delta v_1+(1+\Delta)v_2}} = b^{\frac{v_1+v_2}{2+o(1)}-\frac{1}{2} \parenth{\Delta v_1+(1+\Delta)v_2}}  \\
&\le b^{\frac{1-\Delta}{2} \cdot v_1 - \frac{\Delta}{2} \cdot v_2}\exp(o(n)).
\end{align*}
Hence,
\begin{align*}
\sum_{\pi_2 \in \mc{P}'(r_1,r_2,d_1,d_2)} k^{-n} b^{d(\pi_1,\pi_2)}\le b^{n(1-\rho)\log_b (1-\rho) +\frac{v_1}{2} - \frac{\Delta v}{2}}\exp(o(n)).
\end{align*}
\end{proof}
For the required special case where $\ftn(n)=\epsilon$ for all $n$, we are finally ready to show that the contribution from Case III to (\ref{conversion}) is $o(1)$.
\begin{lemma} If $\ftn(n)=\epsilon$ for all $n$ in the definition of $k$, then
\label{lemmar3}
\[\sum_{\pi_2 \in \mc{P}^{\text{III}}} k^{-n} b^{d(\pi_1,\pi_2)} \exp(o( n)) =o(1).\]
\end{lemma}
\begin{proof}
We will prove that (\ref{yahoo}) is exponentially decreasing in $n$ and will distinguish three cases, depending how large $\Delta-x_0-\epsilon$ is in comparison to $\rho$. Note that 
\begin{equation}
\label{equationnk}
\frac{n}{k}=\gamma-x_0-\epsilon+o(1) = \left \lfloor \gamma \right \rfloor +\Delta-x_0-\epsilon+o(1)=a-1+\Delta-x_0-\epsilon+o(1).
\end{equation}
Roughly speaking, $\Delta-x_0-\epsilon$ is the proportion of parts of size $a$ in a $k$-equipartition, and we need to distinguish between Case 1 where there are few (or no) such parts, Case 2 where there are more such parts but still not so many that all of the $v=\rho n$ vertices in the overlap can be in parts of size $a$, and finally Case 3 where there are enough parts of size $a$ that the overlap blocks between $\pi_1$ and $\pi_2$ can all be in parts of size $a$ in $\pi_1$.  In the first case, we shall only need the condition that $c<\rho<1-c'$; the second and third cases are where condition (\ref{aste}) from Theorem \ref{maintheorem} is crucial.

\begin{itemize}
\item \textbf{Case 1:} $\Delta-x_0-\epsilon<\Delta\rho$.

 If $\frac{n}{k} \le a-1$, then there are no parts of size $a$ in $\pi_1$. 
If $\frac{n}{k} > a-1$, then by (\ref{keyceilnk}), $\ceilnk=a$. Also, $\frac{n}{k}=a-1+\Delta-x_0-\epsilon+o(1)\le a-\epsilon+o(1)<a$, so $\floornk=a-1$. Recall that by (\ref{keyequipartition}) in Section \ref{sectionkeyfacts}, $k_1=\delta k$ where $\delta= \frac{n}{k}-\floornk$. Therefore, if $\frac{n}{k} > a-1$, it follows from (\ref{equationnk}) that $\delta= \Delta-x_0-\epsilon+o(1) \le \Delta \rho+o(1)$. Hence in this case there are $k_1=\delta k \le \Delta\rho k + o(k)$ parts of size $a$ in $\pi_1$, so $v_1\le k_1a \le \Delta \rho n+o(n)$.

In both cases, from (\ref{yahoo}),
 \begin{align}
\sum_{\pi_2 \in \mc{P}'(r_1,r_2,d_1,d_2)} k^{-n} b^{d(\pi_1,\pi_2)}&\le b^{n(1-\rho)\log_b (1-\rho) +\frac{\Delta \rho n}{2} - \frac{\Delta}{2}  \rho n}\exp(o(n))\nonumber\\
&=  b^{n(1-\rho)\log_b (1-\rho) }\exp(o(n))\le b^{-c_6n} \exp(o(n)),\label{part0yahoo}
\end{align}
where $c_6 := \min\parenth{-(1-c) \log (1-c), -c' \log c'}>0$, since $c <\rho \le 1-c'$.

 \item \textbf{Case 2:} $\Delta \rho \le \Delta-x_0-\epsilon \le \rho$.

As $\rho >c$ and $\Delta \ge \epsilon+x_0+\Delta\rho\ge \epsilon$, we have that  $\Delta \rho \ge c \epsilon$. Therefore, by (\ref{equationnk}) and as $\Delta \le 1$,
\[
a-1+c\epsilon +o(1) \le \frac{n}{k}  \le a-\epsilon+o(1).
\]
In particular, $\floornk=a-1$. By (\ref{keyequipartition}) and (\ref{equationnk}), $\pi_1$ has $k_1=\delta k=(\Delta-x_0-\epsilon+o(1))k$ parts of size $a$. Therefore, $v_1$ can be at most $(\Delta-x_0-\epsilon+o(1)) k a$, and as $ka \sim n$,
\begin{align*}
\frac{v_1}{2}-\frac{\Delta v}{2}&\le  \frac{\Delta-x_0-\epsilon}{2}  n -\frac{\Delta}{2} \rho n+o(n) = n \left( \frac{\Delta}{2}\left(1-\rho \right)-\frac{x_0+\epsilon}{2}\right)+o(n).
\end{align*}
Hence, by (\ref{yahoo}),
\begin{align*}
\sum_{\pi_2 \in \mc{P}'(r_1,r_2,d_1,d_2)} k^{-n} b^{d(\pi_1,\pi_2)}\le b^{n\left((1-\rho)\log_b (1-\rho) +\frac{\Delta}{2}\left(1-\rho \right)-\frac{x_0+\epsilon}{2}\right)}\exp(o(n)).
\end{align*}
As remarked above, $\Delta \rho \ge c\epsilon$, so $\Delta - \epsilon-x_0 \ge c\epsilon$. Therefore, we can apply Lemma \ref{technicallemma4} with $\epsilon'=c\epsilon$ and $c_4 = c_4 (\epsilon, c\epsilon)$ to conclude that
\begin{align}
\sum_{\pi_2 \in \mc{P}'(r_1,r_2,d_1,d_2)} k^{-n} b^{d(\pi_1,\pi_2)}&\le b^{-c_4 n} \exp(o(n)). \label{part2yahoo}
\end{align}
Note that the proof of Lemma \ref{technicallemma4} requires Lemma \ref{technicallemma3}, which in turn uses condition (\ref{aste}) from Theorem \ref{maintheorem}.

 \item \textbf{Case 3:} $\Delta-x_0- \epsilon> \rho$.

Noting that $v_1+v_2=v=\rho n$, we proceed from (\ref{yahoo}).
\begin{align*}
\sum_{\pi_2 \in \mc{P}'(r_1,r_2,d_1,d_2)} k^{-n} b^{d(\pi_1,\pi_2)}&\le b^{n(1-\rho)\log_b (1-\rho) +\frac{1-\Delta}{2} v}\exp(o(n))\\
&= b^{n  \left((1-\rho)\log_b (1-\rho) +\frac{1-\Delta}{2}\rho \right)} \exp(o(n)).
\end{align*}
Since $c\le \rho \le \Delta-x_0-\epsilon$, we can use Lemma \ref{technicallemma3} (the proof of which uses condition (\ref{aste})) with $\epsilon'=c$ to see that this expression is exponentially decreasing in $n$.
\begin{align}
\sum_{\pi_2 \in \mc{P}'(r_1,r_2,d_1,d_2)} k^{-n} b^{d(\pi_1,\pi_2)}&\le b^{- {c_3} n} \exp(o(n)). \label{part1yahoo}
\end{align}

\end{itemize}
By (\ref{part0yahoo}), (\ref{part2yahoo}) and (\ref{part1yahoo}), if we let $c_7 = \min (c_3, c_4, c_6) >0$, then 
\begin{align*}
\sum_{\pi_2 \in \mc{P}'(r_1,r_2,d_1,d_2)} k^{-n} b^{d(\pi_1,\pi_2)}&\le b^{-c_7 n} \exp(o(n)).
\end{align*}
Since there are only $O(n^4 \log^2 n)$ choices for the values of $v_i \le n$ and $d_i \le f = O(n\log n)$ for $i=1,2$, this implies
\begin{equation*}
\sum_{\pi_2 \in \mc{P}^{\text{III}} } k^{-n} b^{d(\pi_1,\pi_2)} \exp(o(\log n)) =o(1).
\end{equation*}
\end{proof}
From Lemmas \ref{lemmar2} and \ref{lemmar3} together with (\ref{conversion}), it follows that if $\ftn(n) = \epsilon$ for all $n$, then
\[\sum_{\bfr \in \mc{R}_2^{c'}} Q_{\bfr} b^d=o(1),\]
as required.

\subsection{Very high overlap}
\label{upper}
We are left with those overlap sequences $\bfr$ where $\rho =v/n>1-{c'}$ and $\sum_{2 \le i \le 0.6\gamma} i r_i \le {c'} n$ for any constant ${c'}\in(0,1)$ of our choosing. This means that all but at most $c'n$ vertices are involved in the overlap, and of those vertices involved in the overlap, all but at most $c'n$ are in large overlap blocks of size at least $0.6 \gamma$.  Roughly speaking, in this case the large overlap blocks are mostly just permuted amongst themselves, and there are a small number of exceptional vertices which need to be studied in more detail. Let
\[
 \mc{R}_3^{c'} = \left\lbrace \mathbf{r} \mid \rho > 1-{c'}, \sum_{2 \le i \le 0.6\gamma} i r_i \le {c'} n\right\rbrace.
\]
We will show that if we pick ${c'} >0$ small enough and if $\ftn(n)=\epsilon$ for all $n$, then the contribution from $\mc{R}_3^{c'}$ to the sum (\ref{thesum}) is $o(1)$.  We will pick ${c'}>0$ later in this section, and to ensure this is not circular, we will take care that none of the implicit constants in our $O$-notation depend on~${c'}$.

As in the previous section, let $\pi_1$ be an \emph{arbitrary fixed ordered $k$-equipartition}. Recall that for an overlap sequence $\bfr$, we denote by $P'_{\mathbf{r}}$ the number of ordered $k$-equipartitions with overlap sequence $\mathbf{r}$ with $\pi_1$, and that by (\ref{QandP}), $Q_{\mathbf{r}} =\frac{P'_{\mathbf{r}}}{P}  $. 
Let 
\[
\mc{P}_3 = \left\lbrace  \text{ordered $k$-equipartitions $\pi_2$ such that } \bfr (\pi_1,\pi_2) \in \mc{R}_3^{c'} \right \rbrace,
\]
and recall that $\mu_k=Pq^f$ by (\ref{firstmoment}). Then
\begin{equation}
 \sum_{\bfr \in \mc{R}_3^{c'}} Q_{\bfr} b^d =  \sum_{\bfr \in \mc{R}_3^{c'}} \frac{P'_{\mathbf{r}}}{P}  b^{d} = \frac{b^f}{P} \sum_{\bfr \in \mc{R}_3^{c'}} P'_{\mathbf{r}}  b^{d-f}= \frac{1}{\mu_k} \sum_{\bfr \in \mc{R}_3^{c'}} P'_{\mathbf{r}}   b^{-(f-d)}=\frac{1}{\mu_k} \sum_{\pi_2 \in  \mc{P}_3} b^{-(f-d(\pi_1, \pi_2))}\label{bigfinal},
\end{equation}
where $d(\pi_1, \pi_2):= d(\bfr)$ if $\bfr$ is the overlap sequence of $\pi_1$ and $\pi_2$.

Starting with $\pi_1$, we will generate, and count the number of choices for, $\pi_2 \in  \mc{P}_3$. Since $v=\rho n \ge (1-{c'})n$ and $\sum_{2 \le i \le 0.6\gamma} i r_i \le {c'} n$, most of the overlap between $\pi_1$ and $\pi_2$ consists of large overlap blocks which are merely permuted. More specifically, given $\pi_2 \in \mc{P}_3$, we call an overlap block \emph{large} if it contains at least $0.53 \gamma$ vertices, and let
\begin{align*}
L =& \text{ set of large overlap blocks of size at least $0.53 \gamma$}. \\
\intertext{No part of $\pi_1$ can contain more than one large overlap block, and some parts may not contain any large overlap block at all. It will be more important later to talk about the latter type of part, so given $\pi_2 \in \mc{P}_3$, let}
T =&  \text{ set of parts of $\pi_1$ containing no large overlap block.}\\
\intertext{We call a vertex \textit{exceptional} if it is either not in the overlap at all or not in a large overlap block. If $\pi_2 \in \mc{P}_3$, then by definition there are at most $2 {c'} n$ exceptional vertices. We shall distinguish between \emph{three types of exceptional vertices}.  Again given $\pi_2 \in \mc{P}_3$, let
}
S =& \text{ set of exceptional vertices }\\
 S_1 =& \text{ set of exceptional vertices not in parts in $T$, i.e., in parts containing a large overlap}\\
& \text{ block}\\
S_2 =& \text{ set of exceptional vertices in parts  in $T$ which are either not in the overlap at all or}\\
&\text{ in overlap blocks of size at most $\constcone$}\\
S_3 =& \text{ set of exceptional vertices in parts in $T$ which are in overlap blocks of size greater}\\
& \text { than $\constcone$}\\
g =& \text{ number of overlap blocks of vertices in $S_3$.}
\end{align*}
Let $s=|S|$, $s_i=|S_i|$, and $t=|T|$.  Then, as the vertices in parts in $T$ are exactly those in $S_2 \cup S_3$, and since by (\ref{keyceilnk}), $a-3-\epsilon \le \floornk\le \ceilnk\le a$,
\begin{equation}
\label{equfort}
\frac{s_2+s_3}{a} \le t \le \frac{s_2+s_3}{{a-3-\epsilon}}.
\end{equation}
The vertices in $S_3$ are arranged in blocks of size between $\constcone$ and $0.53\gamma$, so
\begin{equation}
\label{equforg}
 \frac{s_3}{0.53\gamma} \le g \le \frac{s_3}{\constcone}.
\end{equation}
Fix $\mathbf{s}=(s_1,s_2,s_3)$, $g$, and $t$ such that $s= s_1+s_2+s_3 \le 2{c'} n$ and (\ref{equfort}) and (\ref{equforg}) hold, and let
\[
  \mc{P} (\mathbf{s}, t,g) =  \big\lbrace \pi_2 \in  \mc{P}_3 \mid \mathbf{s}(\pi_1,\pi_2)=\mathbf{s}, \,\, t(\pi_1,\pi_2)=t, \,\, g(\pi_1,\pi_2)=g\big\rbrace.
\]
Note that
\begin{equation}
\label{p3andptgs}
   \mc{P}_3 = \bigcup_{\mathbf{s}, t, g: s \le 2 {c'} n} \mc{P} (\mathbf{s}, t,g).
\end{equation}
Starting with the fixed partition $\pi_1$ and given $\mathbf{s}$, $g$, $t$, we will generate all $\pi_2 \in   \mc{P} (\mathbf{s}, t,g)$ and sum $b^{-(f-d(\pi_1, \pi_2))}$ to bound the contribution to (\ref{bigfinal}). We will proceed in the following way: first, we choose all three sets of exceptional vertices, bounding the number of choices in Lemma \ref{Part3Lemma1}. Next, we generate $\pi_2$ by permuting the exceptional vertices amongst themselves and then permuting all the parts, taking into account that part sizes may vary between $\ceilnk$ and $\floornk$. The number of ways to generate $\pi_2$ in this way is bounded in Lemma \ref{Part3Lemma2}. Finally, in Lemma \ref{Part3Lemma3}, we will examine how much each exceptional vertex of each type subtracts from the maximum possible number $f$ of shared forbidden edges between $\pi_1$ and $\pi_2$, and we obtain a lower bound for $f-d(\pi_1, \pi_2)$ which will be used afterwards to bound $b^{-(f-d(\pi_1, \pi_2))}$ from above.

\begin{lemma}
 \label{Part3Lemma1}
For fixed $\mathbf{s}=(s_1,s_2,s_3)$, $g$ and $t$ and the fixed partition $\pi_1$, there are at most 
\[
\frac{n^{s_1} k^t 2^{s_3}t^{g}a^{s_3}}{s_1!g!} 
\]
ways to choose the sets $S_1$, $S_2$ and $S_3$ and arrange the vertices in $S_3$ into $g$ overlap blocks.
\end{lemma}
\begin{proof}
We first choose the vertices in $S_1$ and the parts in $T$. For this there are at most
\begin{equation*}
{n \choose s_1}{k \choose t} \le \frac{n^{s_1} k^t }{s_1!}
\end{equation*}
possibilities. Next, we pick the vertices in $S_3$ from within the parts in $T$ along with the $g$ overlap blocks they make up. Since we do not know the exact sizes of these overlap blocks, we first write $s_3$ as an ordered sum of $g$ positive summands, which can be done in ${s_3-1 \choose g-1}$ ways. Next, we decide which of the parts in $T$ each of the $g$ blocks is in, for which there are at most $t^g$ choices, and then we pick the vertices that belong to each of the $g$ blocks. We know which part of size at most $a$ each such vertex is in, and we choose $s_3$ vertices in total, so there are at most $a ^{s_3}$ possibilities for this. Finally, since we do not care about the order of the $g$ overlap blocks, we can divide by $g!$. So overall, there are at most
\begin{equation*}
{s_3-1 \choose g-1} t^g a ^{s_3} \frac{1}{g!} \le \frac{2^{s_3}t^{g}}{g!} a^{s_3}
\end{equation*}
ways of selecting the vertices in $S_3$ along with the $g$ overlap blocks they are arranged in. The remaining vertices in the parts in $T$ must be exactly those in $S_2$. 
\end{proof}
Let  
\begin{equation}
\tau = \max \left( 1, \Gamma\left(\frac{s_2}{t}\right)^t\right), \label{defoftau}
\end{equation}
where $\Gamma(\cdot)$ denotes the gamma function.
\begin{lemma}
 \label{Part3Lemma2}
Given $\pi_1$, $S_1$, $S_2$, $S_3$ and the overlap blocks that the vertices in $S_3$ are arranged in, there are at most 
\[
 {k \choose k_1} \frac{(s_1+s_2+g)!}{\tau} k_1! k_2!
\]
possibilities for $\pi_2$.
\end{lemma}

\begin{proof}
Note that each part in $\pi_1$ and $\pi_2$ contains at most one large overlap block from $L$, since one such block occupies more than half of a part. Therefore, since we know $S_1$, $S_2$ and $S_3$, we also know $L$. In each part of $\pi_1$, there are a certain number of `slots' for exceptional vertices, with the rest of the part occupied by at most one block from $L$. The numbers of slots for exceptional vertices in parts of $\pi_2$ are essentially just a permutation of the numbers of slots in $\pi_1$, because the remainders of the parts in $\pi_2$ are again occupied by at most one block from $L$. However, as total part sizes vary between $\ceilnk$ and $\floornk$, the numbers of available slots in each part may also increase or decrease by $1$.

Therefore, starting with $\pi_1$, we can generate every possible partition $\pi_2$ in the following way. Each of the $k$ parts of $\pi_1$ contains a certain number of exceptional vertices. We first decide which of the $k$ parts will be of size $\ceilnk$ in $\pi_2$, for which there are ${k \choose k_1}$ possibilities. In each part, this may increase or decrease the number of available slots for exceptional vertices by $1$. We write the vertices in $S_1$ and $S_2$ along with the $g$ blocks comprising the vertices in $S_3$ as a list and permute them, which can be done in 
\[(s_1+s_2+g)!\]
ways. Now we divide up the list successively according to the number of available slots in each of the $k$ parts (discarding the cases where this is not possible because one of the $g$ blocks would have to be divided), and move the vertices from each division to the corresponding part. Finally, we permute all $k_1$ parts of (new) size $\ceilnk$ and all $k_2$ parts of (new) size $\floornk$, for which there are \[
                                                                                                                               k_1! k_2!
                                                                                                                              \]
possibilities, and re-order the parts so that those of size $\ceilnk$ come first, followed by those of size $\floornk$, yielding the new ordered $k$-equipartition $\pi_2$.

However, we have overcounted the number of ways to generate $\pi_2$: each possible partition $\pi_2$ was counted at least $\tau$ times, where $\tau$ is defined in (\ref{defoftau}). To see this, suppose we have generated a partition $\pi_2$. Note that the number of available slots for exceptional vertices in the parts in $T$ is at least $s_2+s_3-t$, since there were initially $s_2+s_3$ exceptional vertices in the parts in $T$, and at most $t$ slots can be `lost'. So at least $s_2+s_3-t$ vertices were moved to the available slots in $T$, and of these, at most $s_3$ were in one of the $g$ overlap blocks. Therefore, there were at least  $s_2-t$ vertices which were permuted and then moved to the parts in $T$ as singletons. Denote the number of such singletons assigned to each of the parts in $T$ by $l_1$, $l_2$, ..., $l_t$, where $\sum_{i=1}^t l_i \ge s_2-t$. Then, since we do not care about the order of the vertices within the parts, we counted $\pi_2$ at least $\prod_{i=1}^t l_i!$ times.

Note that $l_i! = \Gamma (l_i+1)$, where $\Gamma (\cdot)$ denotes the gamma function. By the Bohr--Mollerup Theorem (see for example §13.1.10 in \cite{handbookcomplex}), $\log \Gamma (\cdot)$ is a convex function on the positive reals, so from Jensen's inequality,
\[\log \parenth{\prod_{i=1}^t l_i!} = \sum_{i=1}^t \log (\Gamma(l_i+1))\ge t \log \parenth{\Gamma\left(\frac{1}{t}\sum_{i=1}^t l_i+1\right)}, \]
and therefore $\prod_{i=1}^t l_i!\ge\Gamma\left(\frac{s_2-t}{t}+1\right)^t= \Gamma\left(\frac{s_2}{t}\right)^t$. Hence, we may divide our result by $\tau$.
\end{proof}

\begin{lemma}
 \label{Part3Lemma3} If $\pi_2 \in \mc{P} (\mathbf{s}, t,g)$, then
\[
 f-d(\pi_1, \pi_2) \ge 0.53 {\gamma} s_1 +  \left( \gamma/2-51 \right)s_2+0.23 {\gamma} s_3.
\]
\end{lemma}

\begin{proof}
Note that the number $d(\pi_1, \pi_2)$ of shared forbidden edges is exactly the number of pairs of vertices which  are in the same part in both $\pi_1$ and $\pi_2$, and $f$ is the number of pairs of vertices which are in the same part of $\pi_1$. Therefore, if we let
\[
E= \big\lbrace \left \lbrace v,w\right \rbrace \mid \text{ $v$ and $w$ are in the same part of $\pi_1$ but in different parts of $\pi_2$}\big \rbrace,
\]
then $f-d(\pi_1, \pi_2)=|E|$. Each exceptional vertex $v \in S$ contributes at least a certain amount to $|E|$ according to its type.

If $v$ is in $S_1$, then $v$ is in a part of $\pi_1$ which contains a large overlap block. Therefore, there are at least $0.53 \gamma $ vertices $w \notin S$ such that $\lbrace v,w \rbrace \in E$. Therefore, the contribution from $S_1$ to $|E|$ is at least $0.53 \gamma s_1$.

Since the vertices in $S_2$ are in overlap blocks of size at most $\constcone$, by (\ref{keyceilnk}), for each $v \in S_2$, there are at least $\floornk -\constcone \ge \gamma-\epsilon-103$ vertices $w$ such that $\lbrace v,w \rbrace \in E$. As the vertices in $S_3$ are exceptional and therefore in overlap blocks of size at most $0.53 \gamma$, for each $v \in S_3$, there are at least $\floornk-0.53\gamma \ge 0.46 \gamma$ vertices $w$ such that $\lbrace v,w \rbrace \in E$. However, we have counted each such pair  $\lbrace v,w\rbrace$ twice, and must therefore divide the total number by $2$. So the contribution from $S_2 \cup S_3$ to $|E|$ is at least $\left( {\gamma/2-52-\epsilon/2} \right)s_2+0.23 {\gamma} s_3$.
\end{proof}
By (\ref{key2}) from Section \ref{sectionkeyfacts}, $b^{-\gamma}\sim k^{-2}$, so if $\pi_2 \in  \mc{P} (\mathbf{s}, t,g)$, from Lemma \ref{Part3Lemma3},
\begin{align*}
 b^{-(f-d(\pi_1, \pi_2))} &\le k^{-1.06 s_1 -s_2-0.46 s_3 } \exp\left( O\left( s\right)\right) \le n^{-1.05 s_1} k^{ -s_2-0.46 s_3 } \exp\left( O\left( s\right)\right). 
\end{align*}
Together with Lemmas \ref{Part3Lemma1} and \ref{Part3Lemma2}, this gives
\begin{align}
\sum_{\pi_2 \in \mc{P}(\mathbf{s},t,g)}& b^{-(f-d(\pi_1, \pi_2))} \nonumber \\
&\le  \frac{n^{s_1} k^t 2^{s_3}t^{g}a^{s_3}}{s_1!g!} {k \choose k_1} \frac{(s_1+s_2+g)!}{\tau} k_1! k_2!n^{-1.05 s_1} k^{ -s_2-0.46 s_3 } \exp\left( O\left( s\right)\right)  \nonumber \\
&=  k_1! k_2! {k \choose k_1}n^{-0.05 s_1} t^{g}a^{s_3} \frac{(s_1+s_2+g)!}{s_1!g!\tau}  k^{t -s_2-0.46 s_3 } \exp\left( O\left( s\right)\right)\nonumber \\
&=k_1! k_2! {k \choose k_1}n^{-0.05 s_1} t^{g}a^{s_3} \frac{(s_1+s_2+g)!s_2!}{s_1!s_2!g!\tau}  k^{t -s_2-0.46 s_3 } \exp\left( O\left( s\right)\right). \nonumber
\end{align}
Note that $\frac{(s_1+s_2+g)!}{s_1!s_2!g!} \le 3^{s_1+s_2+g}=   \exp\left( O\left( s\right)\right)$ and by (\ref{equfort}), $s_2! \le s_2^{s_2} \le t^{s_2} a^{s_2}$, so together with~(\ref{equforg}),
\begin{align}
\sum_{\pi_2 \in \mc{P}(\mathbf{s},t,g)} b^{-(f-d(\pi_1, \pi_2))}\le & k_1! k_2! {k \choose k_1}n^{-0.05 s_1} t^{s_2+s_3/\constcone}a^{s_2+s_3} \frac{1}{\tau}  k^{t -s_2-0.46 s_3 } \exp\left( O\left( s\right)\right)\nonumber\\
\le& k_1! k_2! {k \choose k_1}n^{-0.05 s_1} t^{s_2+s_3/\constcone}a^{s_2+s_3} \frac{1}{\tau}  k^{-s_2-0.46 s_3 } \exp\left( O\left( s\right)\right) , \label{biginequ}
\end{align}
as $k^t \le k^{(s_2+s_3)/{(a-3-\epsilon)}}\le \exp(O(s_2+s_3))$ by (\ref{key7}) from Section \ref{sectionkeyfacts}.
Let
\begin{equation*}
T(s_2, s_3) = t^{s_2+s_3/\constcone}a^{s_2+s_3} \frac{1}{\tau}  k^{-s_2-0.46 s_3 } \exp \left( C_2 (s_2+s_3)\right) ,
\end{equation*}
where $C_2>0$ is the constant implicit in the term $ O\left( s\right)$ above. We distinguish two cases.

\begin{itemize}
 \item \textbf{Case $1$:} $s_3 \ge \constcone s_2$.

By (\ref{equfort}), $s_3 \le s_2+s_3\le a t$, so $s_2 \le 0.01at$. Since again by (\ref{equfort}),  $s_2+0.46s_3 \ge 0.46(s_2+s_3) \ge 0.46 {(a-3-\epsilon)} t$, and $t \le k$ and $\tau \ge 1$,
\begin{align}
 T(s_2, s_3) &\le t^{0.02at}a^{a t}  k^{-0.46 (a-3-\epsilon)t } \exp(O(s_2+s_3))\le \parenth {\frac{k^{0.02}   a   }{k^{0.45}}}^{a t}\exp(O(s_2+s_3)) \nonumber \\
&\le n^{-0.4at}  \exp(O(s_2+s_3)) \le  n^{-0.3 (s_2+s_3)}\nonumber 
\end{align}
if $n$ is large enough.

 \item \textbf{Case $2$:} $s_3 < \constcone s_2$.

Then by (\ref{equfort}), $t \le \frac{101s_2}{a-3-\epsilon} \le\frac{51s_2}{\log_b n}$, so $\frac{s_2}{t}\ge \frac{\log_b n}{51}$. By the Stirling approximation of the Gamma function,
\[
 \tau \ge \Gamma \left(s_2/t\right)^t \ge \parenth{\frac{\frac{s_2}{t}-1}{e}}^{s_2-t} \ge (\log n)^{s_2} \exp (O(s_2)).
\]
Furthermore, $t \le \frac{s_2+s_3}{{a-3-\epsilon}}\le \frac{s}{{a-3-\epsilon}} \le \frac{2{c'} n}{{a-3-\epsilon}} \le 3 {c'} k$ if $n$ is large enough, and therefore $\frac{t}{k} \le 3{c'}$ for $n$ large enough. So since $a \le 2 \log_b n = 2 \log n/ \log b$,
\begin{align}
T(s_2, s_3) &\le \left( \frac{t a}{k\log n}\right)^{s_2 } \left(\frac{t^{0.01}a}{k^{ 0.46 }}\right)^{s_3}  \exp(O(s_2+s_3)) \nonumber\\
&\le \left( \frac{6c'}{\log b}\right)^{s_2 } \left(\frac{t^{0.01}a}{k^{ 0.46 }}\right)^{s_3}  \exp(O(s_2+s_3)) \le \parenth{\frac{1}{2}} ^{s_2} n^{-0.3 s_3} \nonumber 
\end{align}
for $n$ large enough if ${c'}>0$ is picked small enough. Pick the constant ${c'} >0$ small enough for this.
\end{itemize}
Therefore in both cases, from (\ref{biginequ}) if $n$ is large enough,
\begin{align}
 \sum_{\pi_2 \in \mc{P}(\mathbf{s},t,g)} b^{-(f-d(\pi_1, \pi_2))} &\le  k_1! k_2! {k \choose k_1}n^{-0.05 s_1} 2 ^{-s_2} n^{-0.3 s_3}\exp(O(s_1)) \nonumber \\
&\le  k_1! k_2! {k \choose k_1}n^{-0.04 s_1} 2 ^{-s_2} n^{-0.3 s_3}. \nonumber
\end{align}
If we sum over $s_1$, $s_2$ and $s_3$ and recall that $g \le s_3/\constcone$ and $t\le s_2+s_3$, by (\ref{p3andptgs}) we get a bound for (\ref{bigfinal}):
\[
 \sum_{\bfr \in \mc{R}_3^{c'}} Q_{\bfr} b^d \le \frac{k_1! k_2! {k \choose k_1}}{\mu_k}\sum_{s_1,s_2,s_3} \bigg(0.01 s_3(s_2+s_3)n^{-0.04 s_1} 2 ^{-s_2} n^{-0.3 s_3} \bigg)  = O \parenth{\frac{k_1!k_2!}{\mu_k}  {k \choose k_1}}.\]
 Let us summarise the calculations from this section in the following lemma.
 \begin{lemma} There is a constant $c'>0$ such that
\begin{equation}   \sum_{\bfr \in \mc{R}_3^{c'}} Q_{\bfr} b^d = O \parenth{\frac{k_1!k_2!}{\mu_k}  {k \choose k_1}}.\label{endofsection}\end{equation}
 \end{lemma}
We now turn to the required special case where $\ftn(n) = \epsilon$ for all $n$. As $k=O\parenth{\frac{n}{\log n}}$,  and since by (\ref{exponentialexpectation}) from Section \ref{sectionkeyfacts}, $\frac{\mu_k}{k_1! k_2!} \ge b^{\epsilon n/4}$ for $n$ large enough, we can see that the right-hand side of (\ref{endofsection}) is exponentially decreasing in $n$, and in particular it is $o(1)$ as required. This concludes the proof of Theorem~\ref{maintheorem}. \qed

\section{Outlook}
Shamir and Spencer \cite{shamir1987sharp} showed that for any function $p=p(n)$, the chromatic number of $\Gnp$ is whp concentrated on an interval of length about $\sqrt{n}$, and for constant $p$, this can be improved to an interval of length about $\sqrt{n}/\log n$ (this is an exercise in Chapter 7.3 of \cite{alonspencer}, see also \cite{scott2008concentration}). However, the proof of this concentration result gives no clue about the location of this interval. While our new explicit bounds on the colouring rate match except for a smaller order additive term, the gap between the corresponding implied chromatic number bounds is still at least of order $\frac{n \log\log n}{\log^3 n}$, which of course is asymptotically larger than $\sqrt{n}/\log n$.

Therefore, a more detailed result on the smaller order additive term would be interesting. For $p\le 1-1/e^2$, where the lower bound comes from the first moment threshold for the number of partitions which induce proper colourings \cite{panagiotou2009note}, this gap is unavoidable as long as the upper bound is obtained through the study of partitions which induce balanced colourings, since the first moment thresholds of colourings and balanced colourings are separated by this distance.

For functions $p(n)$ which tend to $0$ sufficiently quickly, much sharper concentration results are known. In particular, for any $\epsilon>0$ and $p=p(n) \le n^{-1/2-\epsilon}$, the chromatic number of $\Gnp$ is concentrated on at most two values whp \cite{shamir1987sharp,luczak1991note,alon1997concentration}, and this is generally the smallest possible interval one can hope for. In contrast, the question of the concentration of the chromatic number of dense random graphs is wide open. Even the most basic non-concentration results, such as showing that we do not in general have two-point concentration, would be interesting (see also~\cite{bollobas:concentrationfixed}).

\subsubsection*{Acknowledgements}
I am grateful to my supervisor Oliver Riordan for many helpful discussions and for his comments on several earlier versions of this paper. I would also like to thank my viva examiners Colin McDiarmid and Andrew Thomason as well as the two anonymous referees for their suggestions and comments which greatly improved the presentation of this paper.

 \bibliographystyle{plainnat}

\appendix

\section{Appendix}

\counterwithin{theorem}{section} 
\setcounter{theorem}{0}

\counterwithin{equation}{section} 
\setcounter{equation}{0}

\subsubsection*{Proof of Fact (\ref{exponentialexpectation}) from Section \ref{sectionkeyfacts}.}
Note that $k_1! k_2! \le k! \exp\parenth{{o(n)}}$, since ${k \choose k_1} \le 2^k$. Furthermore, by (\ref{key1}) we have $k \sim \frac{n}{2\log _b n}$, so with Stirling's formula, $k! = k^{k(1+o(1))} = n^{k(1+o(1))} =b^{n \parenth{\frac{1}{2}+o(1)}}$. From the definition (\ref{defoff}) of $f$,  $q^f=b^{-f}=b^{-\frac{n^2}{2k}+\frac{n}{2}}\exp(o(n))$. Therefore, from (\ref{pasymp}) and (\ref{firstmoment}),
\[
\frac{\mu_k}{k_1!k_2!} = \frac{P q^{f}}{k_1!k_2!} \ge \frac{ k^n q^f}{k!} \exp \parenth{o(n)} = \parenth{k b^{-\frac{n}{2k}}}^n  \exp \parenth{o(n)}.
\]
Note that $\frac{n}{k}= \gamma-x_0-\ftn+o(1) \le \gamma-\ftn+o(1) $, so by (\ref{key2}),
\[
 k b^{-\frac{n}{2k}} \ge kb^{\frac{-\gamma+\ftn+o(1)}{2}}= b^{\ftn/2+o(1)},
\]
and therefore $\frac{\mu_k}{k_1! k_2!} \ge b^{\ftn n /2} \exp(o(n))$.
\qed

\subsubsection*{}
Recall from Section \ref{technicallemmas} that
\begin{align*}
\varphi(x) &= \varphi_n(x) = (1-\Delta+x) \log_b(1-\Delta+x) +(1-\Delta)(\Delta-x)/2\\
\varphi'(x) &= \log_b (1-\Delta+x)+\frac{1}{\log b}-\frac{1-\Delta}{2}\\
\varphi''(x) &= \frac{1}{(1-\Delta+x)\log b} \ge 0,
\end{align*}
and that $x_0 \in [0, \Delta]$ is the smallest nonnegative solution of $\varphi(x) \le 0$.

\subsubsection*{Proof of Lemma \ref{corollary0}.}
 Note that $p\le1-1/e^2$ is equivalent to $\log b \le 2$, and therefore,
\[
\log_b(1-\Delta)+\frac{\Delta}{2} \le \frac{1}{2} \big(\log (1-\Delta)+\Delta\big) \le 0,
\]
since $\log(1-y) \le -y$ for all $y \in [0,1)$. Hence, $\varphi(0)\le0$. \qed

\subsubsection*{Proof of Lemma \ref{corollarybounds}.}
By definition, $x_0\ge 0$. If $\Delta \le 1-\frac{2}{\log b}$, then $x_0 \le \Delta \le 1-\frac{2}{\log b}$. So suppose $\Delta >1-\frac{2}{\log b}$, then the claim follows if we can show $\varphi\parenth{1-\frac{2}{\log b}}\le 0$. Note that $\varphi\parenth{1-\frac{2}{\log b}}\le 0$ is equivalent to $\psi_1\left(2-\frac{2}{\log b}-\Delta\right)\le 0$, where
\[
 \psi_1(y) = y\log y +\frac{\log b}{2}\parenth{y+\frac{2}{\log b}-1}(1-y).
\]
Note that $\psi_1'(y)=\log y -y \log b+ \log b $ and $\psi_1''(y)= \frac{1}{y}- \log b$.

The function $\psi_1$ has no maximum in $\left( 1-\frac{2}{\log b}, 1\right)$: suppose we have $y \in \left( 1-\frac{2}{\log b}, 1\right)$ with $\psi_1'(y)=0$ and $\psi_1''(y)\le0$. It follows that $0=\log y + (1-y)\log b\ge \log y +\frac{1-y}{y}$, but this is a contradiction since $ \log z +\frac{1-z}{z}>0$ for all $z \in (0,1)$.

In the boundary cases $y=1-\frac{2}{\log b}$ and $y=1$, we have $\psi_1(y) \le 0$. Since $1-\frac{2}{ \log b}<\Delta \le 1$, it follows that $2-\frac{2}{\log b}-\Delta \in \left[ 1-\frac{2}{\log b}, 1\right)$, and therefore $\psi_1 \parenth{2-\frac{2}{\log b}-\Delta}\le 0$ as required.
\qed

\subsubsection*{Proof of Lemma \ref{technicallemma1}.}
 Since $x_0>0$, the definition of $x_0$ implies that $\varphi(x) >0$ for all $x \in [0, x_0)$, so by continuity $\varphi(x_0)=0$ and furthermore $\varphi'(x_0)\le 0$. As $\varphi'' \ge \frac{1}{\log b}$ on $[0, \Delta]$ for all $n$, $\varphi''$ is strongly convex on $[0, \Delta]$ with parameter at least $\frac{1}{\log b}$ for all $n$, and the claim follows.
\qed

\subsubsection*{Proof of Lemma \ref{technicallemma2}.}
Since $\Delta<1$, the claim is trivial if $\epsilon \ge 1$, so suppose $\epsilon \in (0,1)$. As $\varphi(\Delta)=0$ and $\varphi(x_0) \le 0$, there cannot be an $x \in (x_0, \Delta)$ such that $\varphi(x)> 0$, otherwise there would have to be a local maximum which is impossible since $\varphi''>0$.

So $\varphi(\Delta-\epsilon)\le0$ and rearranging terms gives
\[
 1-\Delta \le -\frac{(1-\epsilon) \log (1-\epsilon)}{\epsilon}\cdot\frac{2}{\log b},
\]
so we can let $ c_2 = -\frac{(1-\epsilon) \log (1-\epsilon)}{\epsilon}\in (0,1)$ as $\epsilon \in (0,1)$.
\qed

\subsubsection*{Proof of Lemma \ref{technicallemma3}.}
Let\[
   c_3 = \min \parenth{\frac{\epsilon^2}{4 \log b}, \frac{\epsilon'^2}{4 \log b}} >0     .
       \]
Note that $x_0+\epsilon \le \Delta - y \le  \Delta-\epsilon'$. As $\varphi''>0$, $\varphi$ has no internal maxima in $\left(x_0+\epsilon, \Delta-\epsilon'\right)$, so
\[
 \varphi(\Delta-y) \le \max \parenth{\varphi(x_0+\epsilon), \varphi(\Delta-\epsilon')}.
\]
We distinguish two cases.
\begin{itemize}
 \item \textbf{Case 1: } $\varphi(x_0+\epsilon) \ge \varphi(\Delta-\epsilon')$

Then as $\varphi'$ is increasing, $\varphi' (x_0+\epsilon)=\log_b(1-\Delta+x_0+\epsilon)+\frac{1}{\log b}-\frac{1-\Delta}{2} \le 0$. 
For any $z_1,z_2\ge0$ with $z_1+z_2 <1$, we have $\log(z_1+z_2)\ge \log(z_1)+z_2$, so 
\[\varphi'\parenth{x_0+\frac{\epsilon}{2}} \le \varphi'(x_0+\epsilon)-\frac{\epsilon}{2\log b} \le -\frac{\epsilon}{2\log b}.\]
As $\varphi'$ is increasing, $\varphi'(x)<0$ for all $x \in \left[x_0, x_0+\epsilon \right]$, and since by definition $\varphi(x_0)\le 0$, 
\begin{align*}
\varphi\parenth{x_0+\epsilon} &\le \int_{x_0}^{x_0+\epsilon} \varphi'(x) \,\text{d} x \le  \int_{x_0}^{x_0+\epsilon/2} \varphi'(x) \,\text{d} x \le \frac{\epsilon}{2} \varphi'\parenth{x_0+\frac{\epsilon}{2}} \le -\frac{\epsilon^2}{4 \log b} \le-c_3.
\end{align*}

 \item \textbf{Case 2: } $\varphi(x_0+\epsilon) < \varphi(\Delta-\epsilon')$

Then as $\varphi'$ is increasing, $\varphi' (\Delta-\epsilon') =\log_b(1-\epsilon')+\frac{1}{\log b}-\frac{1-\Delta}{2} \ge 0$. For any $z_1,z_2\ge0$ with $z_1+z_2 <1$, we have $\log(z_1+z_2)\ge \log(z_1)+z_2$, so 
\[\varphi'\parenth{\Delta-\frac{\epsilon'}{2}} \ge \varphi'(\Delta-\epsilon')+\frac{\epsilon'}{2\log b} \ge \frac{\epsilon'}{2\log b}.\]
As $\varphi'$ is increasing, $\varphi'(x)\ge 0$ for all $x \ge \Delta-\epsilon$, and since $\varphi(\Delta)= 0$,
\begin{align*}
 \varphi(\Delta-\epsilon') &=-\int_{\Delta-\epsilon'}^\Delta \varphi'(x) \,\text{d} x \le -\int_{\Delta-\epsilon'/2}^\Delta \varphi'(x) \,\text{d} x \le-\frac{\epsilon'}{2} \varphi'\parenth{\Delta-\frac{\epsilon'}{2}}\\
&\le - \frac{\epsilon'^2}{4\log b} \le-c_3.
\end{align*}
\end{itemize}
\qed

\subsubsection*{Proof of Lemma \ref{technicallemma4}.}
 Let
\[
 \psi_2(x) = (1-x) \log_b (1-x)+\frac{\Delta}{2}(1-x)-\frac{x_0+\epsilon}{2}. 
\]
Note that $\psi_2(\Delta-x_0-\epsilon) =\varphi(x_0+\epsilon) $. Furthermore, $\lim_{x \rightarrow 1}\psi_2(x)=-\frac{x_0+\epsilon}{2} \le -\frac{\epsilon}{2}$. Since $\psi_2''(x) = \frac{1}{(1-x)\log b} >0$ for $x \in(0,1)$, $\psi_2$ has no internal maxima in $(0,1)$, so since $\Delta-x_0-\epsilon\le y \le 1$,
\[
 \psi_2(y) \le \max \parenth{\psi_2(\Delta-x_0-\epsilon), -\frac{\epsilon}{2}} \le \max \parenth{\varphi(x_0+\epsilon), - \frac{\epsilon}{2}}.
\]
Applying Lemma \ref{technicallemma3} to $y'= \Delta-x_0-\epsilon$, we can see that $\varphi(x_0+\epsilon) \le -c_3(\epsilon, \epsilon')$. Letting 
\[
 c_4 = \min \parenth{c_3(\epsilon, \epsilon'), \frac{\epsilon}{2}} >0,
\]
it follows that $\psi_2(y) \le-c_4$ for all $\Delta-x_0-\epsilon\le y \le 1$.
\qed

\subsubsection*{Proof of Lemma \ref{lemmaatmostnconstant}.}
As usual, we assume throughout that $n$ is large enough for our various bounds to hold. First, note that
\begin{equation}
 \frac{T_{i+1}}{T_i} = \frac{e^\rho b^{i+1 \choose 2}\left( a-i\right)!^{2}}{b^{i \choose 2}n (i+1)\left( a-i-1\right)!^{2}} = \frac{e^\rho b^i (a-i)^2}{n(i+1)}. \label{ratioofterms}
\end{equation}
Now consider $i=2$: since $a \sim \frac{n}{k} = O(\log n)$ by (\ref{key3}) in Section \ref{sectionkeyfacts},
\begin{align*}
T_2&= \frac{e^{2\rho }b^{2 \choose 2} k^{2} a!^{2}}{ n^2 2!\left( a-2\right)!^{2}}\le \frac{e^{2 }b k^{2} a^{4}}{ 2n^2 } = O \parenth{\log^2 n}\le n^{1-c_5}.
\end{align*}
By (\ref{ratioofterms}), for $i\le 5$,
\[
T_{i+1} =  O\parenth{ \frac{\log^2 n}{n} } T_i \le n^{-1+o(1)} T_i, 
\]
so in particular for all $3\le i\le 6$,
\[
 T_i \le T_3 \le n^{-1+o(1)} O(\log^2 n) \le n^{-c_5}.
\]
For $7 \le i \le 1.2 \log_b n$, note that as $a \le 2 \log_b n$,
\begin{align*}
T_i \le \frac{e^{ i}b^{\frac{i^2}{2}} n^{2}a^{2i}}{ n^i} = n^2 \parenth {\frac{eb^{\frac{i}{2}}a^2}{n}}^i \le n^2 \parenth{ \frac{4eb^{0.6 \log_b n} \log_b^2 n }{n}  }^i \le n^{2-0.3i} \le n^{-0.1} \le n^{-c_5}.
\end{align*}
For $i\ge1.2 \log_b n$,
\begin{equation}\label{quot1}
 \frac{T_{i+1}}{T_i}  = \frac{e^\rho b^i (a-i)^2}{n(i+1)} \ge n^{0.2+o(1)} \ge 1,
\end{equation}
so for all $1.2 \log_b n \le i \le \ceilnk-1$,
\[
 T_i \le T_{\ceilnk-1}.
\]
So it only remains to show that $T_{\ceilnk-1} \le n^{-c_5}$ and $T_\ceilnk \le n^{1-c_5}$. For this, we first take a look at $T_a$. Since $a=\gamma-\Delta+1$, by (\ref{key2}),
\[
 b^{a \choose 2} = b^{(\gamma-\Delta)(\gamma-\Delta+1)/2}=b^{\frac{\gamma}{2}(\gamma+1-2\Delta)}n^{o(1)} = ((1+o(1))k)^{\gamma+1-2\Delta} n^{o(1)} = k^{a-\Delta}n^{o(1)},
\]
so by Stirling's formula,
\[
 T_a =  \frac{e^{\rho a}b ^{{a \choose 2}} k^2 a!}{ n^a} \sim \frac{e^{ \rho a} k^{a-\Delta +2}n^{o(1)}\sqrt{2\pi a}a^a}{ n^ae^a} \le \parenth{\frac{ka}{n}}^a n^{2-\Delta}e^{-(1-\rho)a}n^{o(1)}.
\]
Since by (\ref{key3}), $a  \sim \frac{n}{k} \sim 2 \log_b n$ and as $\bfr \in \mc{R}_1$, this gives
\[
 T_a \le n^{2-\Delta-(1-\rho)\frac{2}{\log b}+o(1)}\le n^{2-\Delta-(1-c)\frac{2}{\log b}+o(1)}.
\]
For $i \le a-1$ with $a-i=O(1)$, by (\ref{ratioofterms}) and (\ref{key2}) and since $a= \left \lfloor \gamma \right \rfloor+1$,
\begin{equation}\label{ratioaanda}
 \frac{T_{i+1}}{T_{i}} \le \frac{ n^{o(1)}b^i }{n} = n^{1+o(1)}.
\end{equation}
Therefore,
\[
 T_{a-1} = T_a n^{-1+o(1)} \le n^{1-\Delta-(1-c)\frac{2}{\log b}+o(1)}.
\]
To bound $T_\ceilnk$, we need to distinguish between two cases. By (\ref{keyceilnk}), $\ceilnk\le a$.
\begin{itemize}
 \item \textbf{Case 1:} $\ceilnk=a$.
 
It follows that $\frac{n}{k} > a-1 = \left \lfloor \gamma \right \rfloor = \gamma -\Delta$. But since $\ftn(n)=\epsilon$ for all $n$, we also have $\frac{n}{k}\le \gamma-x_0-\epsilon$. Therefore, $x_0+\epsilon \le \Delta$. By Lemma \ref{technicallemma2},
\[
 1-\Delta < \frac{2 c_2}{\log b}.
\]
By the definition (\ref{defofc}) of $c$, $c_2=1-2c$ and since $c_5 \le \frac{c}{2\log b}$,
\[
 T_\ceilnk = T_a \le n^{2-\Delta-(1-c)\frac{2}{\log b}+o(1)}  \le n^{1-\frac{2c}{\log b}+o(1)}\le n^{1-2c_5}.
\]
\item \textbf{Case 2:} $\ceilnk\le a-1$.

By the definition of $c_5 \le \frac{1-c}{2\log b}$ and by (\ref{quot1}),
\[
  T_\ceilnk \le T_{a-1} \le n^{1-(1-c)\frac{2}{\log b}+o(1)}\le n^{1-2c_5}.
\]
\end{itemize}
So in both cases, $T_\ceilnk \le n^{1-2c_5}\le n^{1-c_5}$. By (\ref{ratioaanda}),
\[
 T_{\ceilnk-1} \le n^{-1+o(1)}  T_\ceilnk \le n^{-c_5}.
\]
\qed

\subsubsection*{Proof of Lemma \ref{contributionnottoolarge}.}
For $3 \le i \le \ceilnk-1$, the previous lemma gives
\begin{equation*}
 \frac{1}{r_i!}T_i^{r_i} \le n^{-c_5 r_i} \le n^{-c_5 r_i/2}. 
\end{equation*}
Now suppose $i \in \left \lbrace 2, \ceilnk \right \rbrace$. If $r_i \le \frac{n}{\log^{11} n}$, then
\begin{equation*}
 \frac{1}{r_i!} T_i^{r_i} \le n^{(1-c_5)r_i}\le n^{-c_5 r_i/2}\exp \parenth{\frac{n}{2\log^{10} n}} .
\end{equation*}
Otherwise, if $r_i > \frac{n}{\log^{11} n}$, then since $r_i! \ge r_i^{r_i}/e^{r_i}$,
\begin{equation*}
 \frac{1}{r_i!} T_i^{r_i} \le \parenth{\frac{en^{1-c_5}}{r_i}}^{r_i} \le    \parenth{en^{-c_5} \log^{11} n}^{r_i} \le n^{-c_5r_i/2}.
\end{equation*}
Together with (\ref{eq6}), this gives the result.
\qed

\subsubsection*{Proof of Lemma \ref{implies}.}
By the definition of $\mc{R}_2^{c'}$, it suffices to show that if $\sum_{2 \le i \le 0.6\gamma} i r_i\ge {c'} n$, then \ref{condition1} or \ref{condition2} holds. So suppose that $\sum_{2 \le i \le 0.6\gamma} i r_i\ge {c'} n$. Of those vertices that are in overlap blocks of size at most $0.6 \gamma$, either at least ${c'} n/2$ are in parts of size $a$ or at least ${c'} n/2$ are in parts of size at most ${a-1}$ in $\pi_1$.

So say that at least ${c'} n/2$ of them are in parts of size $a$. In particular, $v_1 \ge {c'} n/2 \ge \frac{n}{\parenth{\log \log n}^2}$. Furthermore, if we denote by $\hat r_i$ the number of overlap blocks of size $i$ in parts of size $a$ in $\pi_1$, then 
\[d_1=\sum_{i=2}^a {i \choose 2}\hat r_i \le 0.3 \gamma  \sum_{2 \le i \le 0.6\gamma} i \hat  r_i + \frac{a-1}{2}\sum_{ 0.6\gamma<i\le a} i \hat  r_i.\]
Since $0.3 \gamma \le \frac{a-1}{2}$ and $\sum_{2 \le i \le 0.6\gamma} i \hat r_i \ge {c'} n / 2$ and $\gamma \le \left\lfloor \gamma \right \rfloor+1 =a$, this is at most
\[ \frac{0.3\gamma c' n}{2} + \frac{a-1}{2}\parenth{v_1-\frac{{c'} n}{2}}\le \frac{a-1}{2}v_1-0.05ac'n.\]
Therefore, 
\[\beta_1 = \frac{2d_1}{v_1 \parenth{a-1}} \le 1-(0.1+o(1))\frac{{c'} n}{v_1}.\]
As $v_1 \le n$,  this is at most 
\[1-0.05 {c'} < 1-\frac{(\log \log n)^4}{\log n},\]
so I holds if $n$ is large enough.

The second case is analogous and implies II.
\qed

\end{document}